\documentclass[11pt]{amsart}
 \usepackage{geometry}                
\usepackage{graphicx}
\usepackage{tikz}
\graphicspath{ {./images/} }
\usepackage{amssymb}
\usepackage{epstopdf}
\usepackage{color}
\usepackage{bbm, dsfont}
\usepackage{enumerate}
\usepackage{etoolbox}
\newtheorem{theorem}{Theorem}[section]
\newtheorem{lemma}[theorem]{Lemma}

\newtheorem{proposition}[theorem]{Proposition}
\newtheorem{Proposition}[theorem]{Proposition}

\theoremstyle{definition}
\newtheorem{definition}[theorem]{Definition}
\theoremstyle{remark}
\newtheorem{remark}[theorem]{Remark}

\theoremstyle{example}

\def\beq{ \begin{equation} }
\def\eeq{ \end{equation} }
\def\mn{\medskip\noindent}

\def\ep{\epsilon}

\def\square{\vcenter{\vbox{\hrule height .4pt
  \hbox{\vrule width .4pt height 5pt \kern 5pt
        \vrule width .4pt} \hrule height .4pt}}}

\def\RR{\mathbb{R}}
\def\ZZ{\mathbb{Z}}

\def\BP{P}

\newcommand{\upspace}{\mathbb{H}}
\newcommand{\harm}{\mathcal{H}}
\newcommand{\sharm}{\bar{\mathcal{H}}}
\newcommand{\asharm}{\widetilde{\mathcal{H}}}
\newcommand{\bharm}{\mathsf{H}}

\numberwithin{equation}{section} 

\begin{document}
\title[]{Stationary Harmonic Measure as the Scaling Limit of Truncated Harmonic Measure}
\author{Eviatar B. Procaccia}
\address[Eviatar B. Procaccia\footnote{Research supported by NSF grant DMS-1812009.}]{Texas A\&M University}
\urladdr{www.math.tamu.edu/~procaccia}
\email{eviatarp@gmail.com}
 
\author{Jiayan Ye}
\address[Jiayan Ye]{Texas A\&M University}
\urladdr{http://www.math.tamu.edu/~tomye}
\email{tomye@math.tamu.edu }

\author{Yuan Zhang}
\address[Yuan Zhang]{Peking University}
\email{zhangyuan@math.pku.edu.cn}

\begin{abstract}
In this paper we prove that the stationary harmonic measure of an infinite set in the upper planar lattice can be represented as the proper scaling limit of the classical harmonic measure of truncations of the infinite set.
\end{abstract}

\maketitle
\tableofcontents
\section{Introduction}
Motivated by the study of Diffusion Limited Aggregation (DLA) on graphs with absorbing boundaries, recently \cite{procaccia2017stationary,procaccia2017sets} an appropriate harmonic measure was defined on the upper planar lattice. The so called stationary harmonic measure is a natural growth measure for DLA in the upper planar lattice. In \cite{procaccia2017stationary} a finite DLA process on the upper planar lattice was defined and studied. Moreover an infinite stationary process that bounds from above any process generated by the stationary harmonic measure was defined. However the most interesting process we wish to study is an infinite stationary DLA. Proving that the stationary DLA is well defined seems to be quite challenging. This paper makes a big contribution in this direction connecting the stationary harmonic measure and the correct scaling of the harmonic measure. A dynamical version of our result will allow the construction of the SDLA. Once defined many geometric results on the SDLA will follow from general theory developed in recent study of stationary aggregation processes e.g. stationary Eden model \cite{Berger-Kagan-Procaccia} and stationary internal DLA \cite{2017stationary}.

\subsection{Notations and Definitions}

Let $\upspace = \{ (x, y) \in \mathbb{Z}^2: y \geq 0 \}$ be the upper half plane including the the x-axis, and $S_n$, $n \geq 0$ be  a $2$-dimensional simple random walk. For any $x \in \upspace$, we write
$$
x = (x^{(1)}, x^{(2)}),
$$
where $x^{(i)}$ denotes the i-th coordinate of $x$. For each $n \geq 0$, define the subsets $L_n \subset \upspace$ as follows:
$$
L_n = \{ (x, n) : x \in \mathbb{Z} \},
$$
i.e. $L_n$ is the horizontal line of height $n$. For each subset $A \subset \upspace$, we define the stopping times
$$
\tau_A = \min \{ n \geq 1: S_n \in A \},
$$
and
$$
\overline{\tau}_A = \min \{ n \geq 0: S_n \in A \}.
$$
 For any $R > 0$, let $B(0,R) = \{ x \in \mathbb{Z}^2 : ||x||_2 < R \}$ be the discrete ball of radius $R$, and abbreviate 
$$
\tau_{R}=\tau_{B(0,R)}, \   \bar\tau_{R}=\bar\tau_{B(0,R)}.
$$
Let $||\cdot||_1$ be the $l_1$ norm. We define
$$
\partial^{\text{out}} A := \{ y \in \upspace \setminus A: \exists x \in A, ||x - y||_1 =1 \}
$$
as the outer vertex boundary of $A$, and define
$$
\partial^{\text{in}} A := \{ y \in A: \exists x \in \upspace \setminus A, ||x - y||_1 =1 \}
$$
as the inner vertex boundary of $A$. Let $P_x(\cdot) = P(\cdot | S_0 = x)$. The stationary harmonic measure $\sharm_A $ on $\upspace$ is introduced in \cite{procaccia2017stationary}. Let $A \subset \upspace$ be a connected set. For any edge $e = (x,y)$ with $x \in A$ and $y \in \upspace \setminus A$, define
$$
\sharm_{A,N}(e) = \sum_{z \in L_N \setminus A} P_z \big( S_{\bar{\tau}_{A \cup L_0}} =x, S_{\bar{\tau}_{A \cup L_0} -1} = y \big).
$$
Note that $\sharm_{A,N}(e) > 0$ if and only if $x \in \partial^{in} A$ and $||x - y||_1 =1$. For all $x \in A$, define
$$
\sharm_{A,N}(x) = \sum_{e \text{ starting from } x} \sharm_{A,N}(e), 
$$
and for all $y \in \upspace \setminus A$, define
$$
\hat{\harm}_{A,N}(y) = \sum_{e \text{ starting in } A \text{ ending at } y} \sharm_{A,N}(e).
$$
\begin{proposition} [Proposition $1$ in \cite{procaccia2017stationary}]
	\label{pp}
For any $A$ and $e$ above, there is a finite $\sharm_A (e)$ such that
$$
\lim_{N \rightarrow \infty} \sharm_{A,N}(e) = \sharm_A (e).
$$
\end{proposition}
$\sharm_A (e)$ is called the stationary harmonic measure of $e$ with respect to $A$. The limits
$$
\sharm_{A}(x) := \lim_{N \rightarrow \infty} \sharm_{A,N}(x) 
$$
and 
$$
\hat{\harm}_{A}(y) := \lim_{N \rightarrow \infty} \hat{\harm}_{A,N}(y) 
$$
also exist, and they are called the stationary harmonic measure of $x$ and $y$ with respect to $A$.

\begin{definition}
	\label{definition growth}
	We say that a set $L_0 \subset A \subset \upspace$ has a polynomial sub-linear growth if there exists a constant $\alpha \in (0,1)$ such that
	$$
	| \{ x = (x^{(1)}, x^{(2)}) \in A: x^{(2)} > |x^{(1)}|^{\alpha} \} | < \infty.
	$$
\end{definition}

In this paper, we write positive constants as $c$, $C$, or $c_0$, but their values can be different from place to place.

\subsection{Main Theorem}

Let $\harm$ be the regular harmonic measure. The main result of this paper proves the asymptotic  equivalence between the stationary harmonic measure of any given point with respect to subset $A$ satisfying Definition \ref{definition growth} and the rescaled regular harmonic measure of the same point with respect to the truncations of $A$. To be precise, 
\begin{theorem}
	\label{theorem part 2}
	For any subset $A$ satisfying Definition \ref{definition growth} and any positive integer $n$, let 
	\beq
	\label{truncation}
	A_n=A \cap \Big\{ [-n, n] \times \mathbb{Z} \Big\}
	\eeq
	be the truncation of $A$ with width $2n$. There is a constant $C\in (0,\infty)$, independent of the set $A$, such that any point $x\in A\setminus L_0$, 
	\beq
	\label{equivalence}
	C \lim_{n\to\infty} n \harm_{A_n}(x)=\sharm_{A}(x). 
	\eeq
	Moreover, $C = 2/ \lim_{n \rightarrow \infty } n \harm_{D_n} (0)$, where $D_n = ([-n,n] \times \{0\}) \cap \mathbb{Z}^2$.
\end{theorem}

\begin{remark}
	For points in $L_0$, we can replace the regular harmonic measure $\harm_{A_n}(x)$ in \eqref{equivalence} by its edge version. I.e.,  we have for all $x\in L_0$, 
	\beq
	\label{equivalence 0}
	C \lim_{n\to\infty} \lim_{\|y\|\to\infty} n P_y \left(S_{\tau_{A_n}}=x, S^{(2)}_{\tau_{A_n}-1}>0\right)=\sharm_{A}(x). 
	\eeq
	Later one can see the proof of \eqref{equivalence 0} follows exactly the same argument as the one for \eqref{equivalence}. 
\end{remark}

The structure of this paper is as follows: We show that stationary harmonic measure is equivalent to a normalized harmonic measure in section \ref{s1}, and the proof of Theorem \ref{theorem part 2} is presented in section \ref{s2}. 

\section{Stationary Harmonic Measure is Equivalent to Normalized Harmonic Measure}
\label{s1}

\begin{lemma}
\label{l1}
For all $x \in L_0$, $\sharm_{L_0} (x) = 1$.
\end{lemma}
 
\begin{proof}
Like Proposition $1$ in \cite{procaccia2017stationary}, the proof follows a coupling argument by translating one path starting from a fixed point of $L_N$ horizontally. For each $N$, let $S_n^{(0,N)}$ be a simple random walk in the probability space $P_{(0,N)} (\cdot)$ starting at $(0,N)$, and $S_n^{(k,N)} = S_n^{(0,N)} + (k,0)$ for all $k \in \mathbb{Z}$. Note that $S_n^{(k,N)}$ is a simple random walk starting at $(k,N)$. Let
$$
\overline{\tau}_{L_0} = \inf \{n \geq 0: S_n^{(0,N)} \in L_0 \} 
$$
be a stopping time. Then we have
$$
\overline{\tau}_{L_0} = \inf \{n \geq 0: S_n^{(k,N)} \in L_0 \} 
$$
for any $k \in \mathbb{Z}$, and
$$
S_{\overline{\tau}_{L_0}}^{(k,N)} = S_{\overline{\tau}_{L_0}}^{(0,N)} + (k,0).
$$
Hence,
$$
\sharm_{L_0, N} (x) = \sum_{k \in \mathbb{Z} } P( S_{\overline{\tau}_{L_0}}^{(k,N)} = x) = 1.
$$
By definition of the stationary harmonic measure,
$$
\sharm_{L_0} (x) = \lim_{N \rightarrow \infty} \sharm_{L_0, N} (x) =1.
$$

\end{proof}

We now define a new measure $\asharm_A(\cdot)$ which can be shown equivalent to the stationary harmonic measure $\sharm_A (\cdot)$. For each $n > 0$, we first define
$$
\asharm_{A,n} (x) = \pi n P_{(0,n)} (S_{\tau_{A \cup L_0}} = x).
$$

\begin{lemma}
\label{l2}
For all $x = (x^{(1)}, 0) \in L_0$,
$$
\lim_{n \rightarrow \infty} \asharm_{L_0 ,n} (x) = 1.
$$
\end{lemma}

\begin{proof}
By Theorem 8.1.2 in Lawler and Limic \cite{lawler2010random},
$$
P_{(0,n)} (S_{\tau_{L_0}} = x) = \dfrac{n}{\pi (n^2 + (x^{(1)})^2)} \Bigg( 1 + O \bigg(  \dfrac{n}{n^2 + (x^{(1)})^2} \bigg) \Bigg) + O \Bigg( \dfrac{1}{(n^2 + (x^{(1)})^2)^{3/2}} \Bigg). 
$$
So,
$$
\lim_{n \rightarrow \infty} \asharm_{L_0 ,n} (x) = 1.
$$
\end{proof}

Similar to the construction of the stationary harmonic measure $\sharm_A (\cdot)$, we want to define a measure $\asharm_{A}$ on $\upspace$ as following:
$$
\asharm_{A} (x) := \lim_{ N \rightarrow \infty } \asharm_{A,N} (x),
$$
and denote it by the {\bf $in$-harmonic measure}. We want to show that $ \asharm_{A} = \sharm_A $. We already proved that $\asharm_{L_0} = \sharm_{L_0}$ in Lemma \ref{l1} and Lemma \ref{l2}.
 
\begin{proposition}
Let $A \subset \upspace$ be a connected finite subset. For any $x \in \upspace$,
$$
\asharm_{A} (x) := \lim_{ N \rightarrow \infty } \asharm_{A,N} (x)
$$
exists, and $\asharm_{A} (x) = \sharm_A (x)$.
\end{proposition}

\begin{proof}
Without loss of generality, we assume $x \in \partial^{out} A$. Let 
$$
k = \max \{ x^{(2)} : x = (x^{(1)}, x^{(2)}) \in A \},
$$
and $ n > m > k$ so that $L_m \cap A = \emptyset $. By Strong Markov Property and translation invariance of simple random walk,
\begin{equation}
\begin{split}
& \asharm_{A,n} (x) \\
& = \pi n P_{(0,n)} (S_{\tau_{A \cup L_0}} = x) \\
& = \pi n \sum_{ y \in L_m } P_{(0,n)} (S_{\tau_{L_m}} = y) P_y (S_{\tau_{A \cup L_0}} = x) \\
& = \dfrac{n}{n-m} \sum_{ y \in L_m } P_y (S_{\tau_{A \cup L_0}} = x) \bigg[ \pi (n-m) P_{(0,n)} (S_{\tau_{L_m}} = y) \bigg] \\
& = \dfrac{n}{n-m} \sum_{ y \in L_m } P_y (S_{\tau_{A \cup L_0}} = x) \asharm_{L_0 ,n - m} (y_0),
\end{split}
\end{equation}
where $y_0 = (y^{(1)},0)$. Then by Dominated Convergence Theorem and Lemma \ref{l2},
\begin{equation}
\label{666}
\begin{split}
& \lim_{ n \rightarrow \infty } \asharm_{A,n} (x) \\
& = \lim_{ n \rightarrow \infty } \sum_{ y \in L_m } P_y (S_{\tau_{A \cup L_0}} = x) \dfrac{n}{n-m} \asharm_{L_0 ,n - m} (y_0) \\
& = \sum_{ y \in L_m } P_y (S_{\tau_{A \cup L_0}} = x) \bigg[ \lim_{ n \rightarrow \infty } \dfrac{n}{n-m} \asharm_{L_0 ,n - m} (y_0) \bigg] \\
& = \sum_{ y \in L_m } P_y (S_{\tau_{A \cup L_0}} = x) \\
& = \sharm_{A, m} (x).
\end{split}
\end{equation}
We can apply Dominated Convergence Theorem in equation (\ref{666}) because $\asharm_{L_0 ,n - m} (y_0)$ is uniformly bounded from above for all $n$ and $y_0 \in \mathbb{Z}$ by Theorem 8.1.2 of \cite{lawler2010random} and the fact that $\asharm_{L_0 ,n - m} (0) \geq \asharm_{L_0 ,n - m} (y_0)$ for all $y_0 \in \mathbb{Z}$. We claim that $\sharm_{A, m} (x) = \sharm_{A} (x)$. Let $m_1 > m$. By Strong Markov Property and Lemma \ref{l1},
\begin{equation}
\begin{split}
& \sharm_{A, m_1} (x) \\
& = \sum_{ y \in L_{m_1} } P_y (S_{\tau_{A \cup L_0}} = x) \\
& = \sum_{ y \in L_{m_1} } \sum_{ z \in L_{m} } P_y (S_{\tau_{L_{m} }} = z) P_z (S_{\tau_{A \cup L_0}} = x) \\ 
& = \sum_{ z \in L_{m} } P_z (S_{\tau_{A \cup L_0}} = x) \bigg[ \sum_{ y \in L_{m_1} } P_y (S_{\tau_{L_{m} }} = z) \bigg] \\
& = \sum_{ z \in L_{m} } P_z (S_{\tau_{A \cup L_0}} = x) \sharm_{L_0, m_1 - m} (z') \\
& = \sum_{ z \in L_{m} } P_z (S_{\tau_{A \cup L_0}} = x) \\
& = \sharm_{A, m} (x),
\end{split}
\end{equation}
where $z' = z - (0, m)$. Hence,
$$
\asharm_{A} (x) = \sharm_{A, m} (x) = \lim_{N \rightarrow \infty } \sharm_{A, N} (x) = \sharm_{A} (x).
$$
\end{proof}

Our next goal is to show that the measures $\asharm_{A}$ and $\sharm_A$ are equivalent for sets that satisfy polynomial sub-linear growth condition. We first prove the following combinatorial result: For any positive integer $n$, consider the following rectangle in $\ZZ^2$: 
\beq
\label{set A_n}
I_n =[-n,n]\times[0,n]
\eeq
with height $n$ and width $2n$. It is easy to see that $I_n\subset B(0,2n)$. Moreover, we let $\partial^{in}I_n$ be the inner vertex boundary of $A_n$, and let 
$$
\partial^{in}_{l} I_n =\{-n\}\times[1,n], \ \ \partial^{in}_{r} I_n=\{n\}\times[1,n], \ \ \partial^{in}_{u} I_n=[-n,n]\times\{n\}, \ \ \partial^{in}_{b} I_n=[-n,n]\times\{0\}
$$
be the four edges of $\partial^{in} I_n$. 

Let $\{S_n, \ n\ge 0\}$ be a simple random walk starting from $0$ and denote by $P_0$ the probability distribution of $S_n$. Define stopping time 
$$
T_n=\inf\{k>0, \ S_k\in \partial^{in}I_n\}.
$$
Using simple combinatorial arguments, we prove the following lemma: 
\begin{lemma}
\label{lemma combinatorial}
For any integer $n>1$
$$
P_0\left(S_{T_n}\in \partial^{in}_{u} I_n \right)\ge P_0\left(S_{T_n}\in \partial^{in}_{l} I_n \cup \partial^{in}_{r} I_n \right). 
$$
\end{lemma}
\begin{proof}
Let $\partial^{in}_{u,+} I_n = [1,n]\times\{n\}$ and $\partial^{in}_{u,-} I_n=[-n,-1]\times\{n\}$ be the left and right half of $\partial^{in}_{u} I_n$. By symmetry it suffices to prove that 
\beq
\label{combinatorial 1}
P_0\left(S_{T_n}\in \partial^{in}_{u,+} I_n \right)\ge P_0\left(S_{T_n}\in \partial^{in}_{r} I_n \right). 
\eeq
 By definition, we have 
$$
P_0\left(S_{T_n}\in \partial^{in}_{u,+} I_n \right)=\sum_{k=1}^\infty P_0\left(S_{k}\in \partial^{in}_{u,+} I_n, \ T_n=k \right)
$$
and 
$$
P_0\left(S_{T_n}\in \partial^{in}_{r} I_n \right)=\sum_{k=1}^\infty P_0\left(S_{k}\in \partial^{in}_{r} I_n, \ T_n=k \right).
$$
Moreover, for each $k$, 
$$
P_0\left(S_{k}\in \partial^{in}_{u,+} I_n , \ T_n=k \right)=\frac{|\mathcal{U}^+_{n,k}|}{4^k}, \ \ P_0\left(S_{k}\in \partial^{in}_{r} I_n, \ T_n=k \right)=\frac{|\mathcal{R}_{n,k}|}{4^k}
$$
where 
$$
\begin{aligned}
\mathcal{U}^+_{n,k}=\left\{(a_0,a_1,\cdots, a_k), \ \text{such that } a_0=0, \ \|a_{i+1}-a_i\|=1, \ \forall i=0,1,\cdots, k-1, \right.\\
\left. \ \ a_j\in A_n\setminus\partial^{in}A_n, \forall j=1,2,\cdots, k-1,  \ a_k\in \partial^{in}_{u,+} I_n  \right\}
 \end{aligned}
$$
and 
$$
\begin{aligned}
\mathcal{R}_{n,k}=\left\{(a_0,a_1,\cdots, a_k), \ \text{such that } a_0=0, \ \|a_{i+1}-a_i\|=1, \ \forall i=0,1,\cdots, k-1, \right.\\
\left. \ \ a_j\in A_n\setminus\partial^{in}A_n, \forall j=1,2,\cdots, k-1,  \ a_k\in \partial^{in}_{r} I_n  \right\}
 \end{aligned}
$$
give the subsets of the random walk trajectories in events $\{S_{T_n}\in \partial^{in}_{u,+} I_n\}$ and $\{S_{T_n}\in \partial^{in}_{r} I_n\}$.

Thus in order to show \eqref{combinatorial 1}, we construct a one-to-one mapping $\varphi$ between the trajectories in $\mathcal{R}_{n,k}$ and $\mathcal{U}^+_{n,k}$. For any trajectory $\vec a=(a_0,a_1,\cdots, a_k)\in \mathcal{R}_{n,k}$, define 
$$
m(\vec a)=\sup\left\{i\ge 0, \ a_i^{(1)}=a_i^{(2)}\right\}
$$
to be the last point in the trajectory lying on the diagonal. Here $a_i^{(1)}$ and $a_i^{(2)}$ are the two coordinates of $a_i$. In this paper, we use the convention that $\sup\{\emptyset\}=-\infty$. Then it is easy to see that $0\in \left\{i\ge 0, \ a_i^{(1)}=a_i^{(2)}\right\}$ and thus $m(\vec a)\ge 0$ and that $m(\vec a)<k$. The reason of the latter inequality is that suppose $m(\vec a)=k$, then we must have $a_k=(n,n)$ which implies that $a_{k-1}=(n-1,n)$ or $(n,n-1)$, which contradicts with the definition of $\vec a$. 

Now we can define 
$$
\varphi(\vec a)=\vec a'=(a_0',a_1',\cdots, a_k')
$$
such that 
\begin{itemize}
\item $a_i'=a_i$ for all $i\le m(\vec a)$. 
\item $a_i'=\left(a_i^{(2)},a_i^{(1)}\right)$ for all $i> m(\vec a)$. 
\end{itemize}

\begin{figure}[h]
\centering 
\begin{tikzpicture}[scale=0.5]
\tikzstyle{redcirc}=[circle,
draw=black,fill=red,thin,inner sep=0pt,minimum size=2mm]
\tikzstyle{bluecirc}=[circle,
draw=black,fill=blue,thin,inner sep=0pt,minimum size=2mm]

\draw [black,thin] (10,0) to (0,0);

\draw[black,thin] (10,0) to (20,0);

\draw [black,thin] (0,0) to (0,10);

\draw [black,thin] (20,0) to (20,10);

\draw [black,thin] (0,10) to (20,10);

\node (v1) at (10,0) [bluecirc] {};

\draw [black,dashed] (10,0) to (20,10);

\draw [blue,thick] plot [smooth, tension=2.8] coordinates {(10,0) (9,3) (11,2.5) (12,4) (13,3)};

\draw [blue,thick] plot [smooth, tension=2.5] coordinates {(13,3) (15,2.5) (14,2) (15,6)};

\draw [blue,thick] plot [smooth, tension=2.5] coordinates {(15,6) (15.5,3) (15.8,6) (17,7)};

\draw [blue,thick] plot [smooth, tension=2.5] coordinates {(17,7) (18,6.5) (18.8,8.3) (19.5,8.2) (20,8.1)};

\draw [red,dashed] plot [smooth, tension=2.5] coordinates {(17,7) (16.5,8) (18.3,8.8) (18.2,9.5) (18.1,10)};

\node (v2) at (20,8.1) [bluecirc] {};

\node (v3) at (18.1,10) [redcirc] {};

\node (v4) at (17,7) [bluecirc] {};

\end{tikzpicture}
\caption{mapping between trajectories in $\mathcal{R}_{n,k}$ and $\mathcal{U}^+_{n,k}$}
\end{figure}
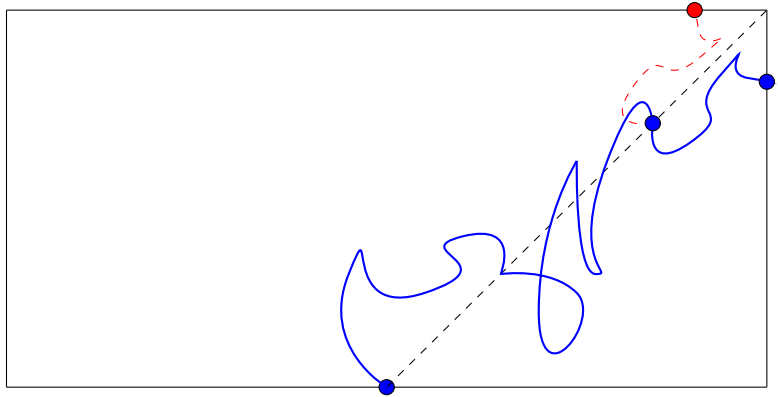

I.e., we reflect the trajectory after the last time it visits the diagonal line $x=y$. By definition
$$
\left(a_{m(\vec a)+1}, a_{m(\vec a)+2},\cdots, a_{k-1}\right)
$$
stays within $\{(x,y)\in\ZZ^2, \  0< y<x< n\}$, while $a_k\in R_{n}$. Thus, under reflection we have 
$$
\left(a_{m(\vec a)+1}', a_{m(\vec a)+2}',\cdots, a_{k-1}'\right)
$$
stays within $\{(x,y)\in\ZZ^2, \  0< x<y< n\}$, while $a_k'\in U^+_{n,k}$, which implies that $\vec a'\in \mathcal{U}^+_{n,k}$. 

On the other hand, suppose we have two trajectories $\vec a$ and $\vec b$ both in $ \mathcal{R}_{n,k}$, such that $\varphi(\vec a)=\varphi(\vec b)$. Then one must have $m(\vec a)=m(\vec b)=m$ and that $a_i=b_i$ for all $i \le m$. Moreover, for all $i>m$, we have 
$$
\left(a_i^{(2)},a_i^{(1)}\right)=a_i'=b_i'=\left(b_i^{(2)},b_i^{(1)}\right)
$$
which also implies that $a_i=b_i$. Thus we have shown that $\varphi(\vec a)=\varphi(\vec b)$ if and only if $\vec a=\vec b$ and $\varphi$ is a one-to-one mapping, which conclude the proof of this lemma. 
\end{proof}

We define
$$
F_m = F_{m, \alpha } = \{ - \lfloor m^{1/\alpha} \rfloor , \lfloor m^{1/\alpha} \rfloor \} \times \mathbb{Z}_{\geq 0}
$$
as two vertical lines on $\upspace$. 

\begin{lemma}
\label{l4}
Fix $x \in \upspace$, then for all sufficiently large $m$,
$$
P_x (\tau_{F_{m, \alpha }} < \tau_{L_0} ) \leq cm^{-1/\alpha}.
$$
\end{lemma}

\begin{proof}
Let $m > 4|x_1|$, and $x' = (x^{(1)}, 0)$. There exists a constant $C >0$ independent of $m$ such that
$$
C P_x (\tau_{F_{m, \alpha }} < \tau_{L_0} ) \leq P_{x'} (\tau_{F_{m, \alpha }} < \tau_{L_0} ).
$$
By translation invariance of simple random walk, we have 
$$
P_{x'} (\tau_{F_{m, \alpha }} < \tau_{L_0} ) \leq P_{0} (\tau_{I_{\lfloor m^{1/\alpha} /2 \rfloor}} < \tau_{L_0} ).
$$
By Lemma \ref{lemma combinatorial},
$$
P_{0} (\tau_{I_{\lfloor m^{1/\alpha} /2 \rfloor}} < \tau_{L_0} ) \leq 2 P_{0} (\tau_{L_{\lfloor m^{1/\alpha} /2 \rfloor}} < \tau_{L_0} ) \leq cm^{-1/\alpha}.
$$
\end{proof}

The next lemma claims that $\sharm_A$ is concentrated on the part arising from random walks starting from $y \in L_m$ such that $|y^{(1)}| \leq \lfloor m^{1/ \alpha} \rfloor$.

\begin{lemma}
\label{l5}
Let $A \subset \upspace$ be an infinite set that has polynomial sub-linear growth with parameter $ \alpha \in (0,1)$. Let $1 > \alpha_1 = ( \alpha +1 )/2 > \alpha$, then for any $x \in \upspace$,
$$
\lim_{m \rightarrow \infty} \Bigg| \sum_{y \in L_m ,  |y^{(1)}| \leq \lfloor m^{1/ \alpha_1} \rfloor } P_y (S_{\tau_{A} \cup L_0} =x) - \sharm_{A,m} (x) \Bigg| = 0. 
$$
\end{lemma}

\begin{proof}
Note that $ \{ y \in L_n , |y^{(1)}| \leq \lfloor n^{1/ \alpha_1} \rfloor \} \cap A = \emptyset$.  Following the argument in \cite[Lemma 2]{kesten1987hitting} on time reversibility and symmetry of simple random walk, we have

\begin{equation}
\begin{split}
& P_y ( \tau_x = k \text{, } S_1 \text{, } \cdot \cdot \cdot \text{, } S_{k-1} \notin \{ x \} \cup L_0 ) \\
& = P_x ( \tau_y = k \text{, }  S_1 \text{, } \cdot \cdot \cdot \text{, } S_{k-1} \notin \{ x \} \cup L_0 ) \\
& = P_x ( S_k = y \text{, } \tau_{ \{ x \} \cup L_0 } > k).
\end{split}
\end{equation}
Then taking the summation over all $k$, we have
\begin{equation}
\begin{split}
& P_y ({\tau_{x}} \leq \tau_{L_0}) \\ 
& = \sum_{k = 1}^{\infty} P_y ( \tau_x = k \text{, } S_1 \text{, } \cdot \cdot \cdot \text{, } S_{k-1} \notin \{ x \} \cup L_0 ) \\
& = \sum_{k = 1}^{\infty} P_x ( S_k = y \text{, } \tau_{ \{x\} \cup L_0 } > k) \\
& \leq E_x \Bigg[ \text{ number of visits to } y \text{ in the time interval } [0, \tau_{ \{x\} \cup L_0 } ) \Bigg] \\
& \leq E_x \Bigg[ \text{ number of visits to } y \text{ in the time interval } [0, \tau_{ L_0 } ) \Bigg]
\end{split}
\end{equation}
Then,
\begin{equation}
\begin{split}
& \lim_{m \rightarrow \infty} \sum_{y \in L_m \setminus A,  |y^{(1)}| \geq \lceil m^{1/\alpha_1} \rceil } P_y (S_{\tau_{A}} =x) \\
& \leq \lim_{m \rightarrow \infty} \sum_{y \in L_m \setminus A,  |y^{(1)}| \geq \lceil m^{1/\alpha_1} \rceil } P_y ({\tau_{x}} \leq \tau_{L_0} ) \\
& \leq \lim_{m \rightarrow \infty} \sum_{y \in L_m \setminus A,  |y^{(1)}| \geq \lceil m^{1/\alpha_1} \rceil } E_x \Bigg[ \text{ number of visits to } y \text{ in the time interval } [0, \tau_{L_0}) \Bigg] \\
& \leq \lim_{m \rightarrow \infty} E_x \Bigg[ \text{ number of visits to } G_{m,\alpha_1} \text{ in the time interval } [0, \tau_{L_0}) \Bigg], 
\end{split}
\end{equation}
where $G_{m,\alpha_1} = \{ y \in L_m : |y^{(1)}| \geq \lceil m^{1/\alpha_1} \rceil \}$. By Lemma \ref{l4}, we have
\begin{equation}
\begin{split} 
& \lim_{m \rightarrow \infty} \Bigg| \sum_{y \in L_m \setminus A,  |y^{(1)}| \geq \lceil m^{1/\alpha_1} \rceil } P_y (S_{\tau_{A}} =x) \Bigg| \\
& \leq \lim_{m \rightarrow \infty} E_x \Bigg[ \text{ number of visits to } G_{m,\alpha_1} \text{ in the time interval } [0, \tau_{L_0}) \Bigg] \\
& \leq \lim_{m \rightarrow \infty} 4m P_x (\tau_{G_{m,\alpha_1}} < \tau_{L_0} ) \\
& \leq \lim_{m \rightarrow \infty} 4m P_x (\tau_{F_{m, \alpha_1 }} < \tau_{L_0} ) \\
& = 0.
\end{split}
\end{equation}
The proof is complete.
\end{proof}

\begin{lemma}
\label{l6}
Let $A \subset \upspace$ be an infinite set that has polynomial sub-linear growth with parameter $ \alpha \in (0,1)$. Let $1 > \alpha_1 = ( \alpha +1 )/2 > \alpha$, then for all $x \in \upspace$ and for all $\epsilon > 0$ and for $m$ and $n=n(m)$ large enough, we have 
$$
 \Bigg| \sum_{y \in L_m ,  |y^{(1)}| \leq \lfloor m^{1/\alpha_1} \rfloor } P_y (S_{\tau_{A} \cup L_0 } =x) - \asharm_{A,n} (x) \Bigg| < \epsilon.
$$
\end{lemma}

\begin{proof}
Fix $x \in \upspace$ and $\epsilon > 0$. Let $l = \max \{ y^{(2)} : y \in A, y^{(2)} > |y^{(1)}|^{\alpha } \}$. Assume that $n$ and $m$ are large with $n > m > \max \{l, x^{(2)} \}$. Let $\alpha_1 = ( \alpha + 1)/2$ as defined in Lemma \ref{l5}. By strong Markov property, we have
\begin{equation}
\label{777}
\begin{split}
& \asharm_{A,n} (x) \\
& = \pi n P_{ (0,n) } (S_{\tau_{A }} = x) \\
& = \sum_{y \in L_{m} \setminus A} \pi n P_{(0,n)} (S_{\tau_{A \cup L_m}} = y) P_y (S_{\tau_{A}} = x) \\
& \leq \sum_{y \in L_{m} , |y^{(1)}| \leq \lfloor m^{1 / \alpha_1} \rfloor } \pi n P_{(0,n)} (S_{\tau_{A \cup L_m}} = y) P_y (S_{\tau_{A }} = x) +  c \sum_{y \in L_{m} \setminus A , |y^{(1)}| \geq \lceil m^{1/\alpha_1} \rceil } P_y (S_{\tau_{A }} = x),
\end{split}
\end{equation}
where $c > 0$ is a constant. The last inequality of equation (\ref{777}) is using Theorem 8.1.2 in \cite{lawler2010random} and the fact that
$$
P_{(0,n)} (S_{\tau_{A \cup L_m}} = y) \leq P_{(0,n)} (S_{\tau_{L_m}} = y).
$$
By Lemma \ref{l5}, we know
$$
\lim_{m \rightarrow \infty} \sum_{y \in L_{m} \setminus A, |y^{(1)}| \geq \lceil m^{1/ \alpha_1} \rceil } P_y (S_{\tau_{A}} = x) = 0.
$$
So there exists a $M_1 > \max \{l, x^{(2)} \}$ such that for all $m > M_1$ and all sufficiently large $n > m$,
$$
\Bigg| \asharm_{A,n} (x) - \sum_{y \in L_{m} , |y^{(1)}| \leq \lfloor m^{1 / \alpha_1} \rfloor } \pi n P_{(0,n)} (S_{\tau_{A \cup L_m}} = y) P_y (S_{\tau_{A }} = x) \Bigg| < \frac{\epsilon}{2}.
$$
Denote the set 
$$
\widetilde{A}_m = \{ x \in \upspace: x^{(1)} > \lfloor m^{1/\alpha} \rfloor, m \leq x^{(2)} \leq |x^{(1)}|^{\alpha} \}.
$$
Note that $\widetilde{A}_m$ contains the part of $A$ that is above the horizontal line $L_m$. For $y \in L_m$ such that $|y^{(1)}| \leq m^{1/\alpha_1}$,we have
\beq
P_{(0,n)} (S_{\tau_{A \cup L_m}} = y)\le P_{(0,n)} (S_{\tau_{L_m}} = y) 
\eeq
while 
\begin{equation}
\begin{split}
P_{(0,n)} (S_{\tau_{A \cup L_m}} = y) & \geq P_{(0,n)} (S_{\tau_{\widetilde{A}_m \cup L_m}} = y) \\
& = P_{(0,n)} (S_{\tau_{L_m}} = y) - \sum_{z \in \widetilde{A}_m} P_{(0,n)} (S_{\tau_{\widetilde{A}_m \cup L_m}} = z) P_z (S_{\tau_{L_m}} = y).
\end{split}
\end{equation}
Note that for $z \in \widetilde{A}_m$, $P_{(0,n)} (S_{\tau_{\widetilde{A}_m \cup L_m}} = z) = 0$ unless $z$ is in the upper inner boundary of $\widetilde{A}_m$, i.e. $z = (k, \lfloor k^{\alpha} \rfloor ) \in \partial^{\text{in}} \widetilde{A}_m $ for some $k > \lfloor m^{1/\alpha} \rfloor $. Suppose $z = (k, \lfloor k^{\alpha} \rfloor ) \in \partial^{\text{in}} \widetilde{A}_m $ with $k > \lfloor m^{1/\alpha} \rfloor $. Let $y \in L_m$ such that $|y^{(1)}| \leq m^{1/ \alpha_1}$. By Theorem 8.1.2 in Lawler and Limic \cite{lawler2010random}, we have
\begin{equation}
\begin{split}
& P_z (S_{\tau_{L_m}} = y) \\
& \leq \frac{c(\lfloor k^{\alpha} \rfloor - m)}{(\lfloor k^{\alpha} \rfloor - m)^2 + (k- \lfloor m^{1/\alpha_1} \rfloor )^2} \\
& \leq \frac{c( k^{\alpha} - m)}{(\lfloor k^{\alpha} \rfloor - m)^2 + (k- m^{1/\alpha_1} )^2}.
\end{split}
\end{equation}
So,
\begin{equation}
\begin{split}
& \sum_{z \in \widetilde{A}_m} P_{(0,n)} (S_{\tau_{\widetilde{A}_m \cup L_m}} = z) P_z (S_{\tau_{L_m}} = y) \\
& \leq \sum_{z \in \widetilde{A}_m} P_z (S_{\tau_{L_m}} = y) \\
& \leq c \sum_{k = \lceil m^{1/\alpha} \rceil }^{\infty} \frac{ k^{\alpha} - m}{(\lfloor k^{\alpha} \rfloor - m)^2 + (k- m^{1/\alpha_1} )^2} \\
& \leq c \sum_{s = 1}^{\infty} \frac{ (s + m^{1/\alpha} +1)^{\alpha} - m}{( \lfloor (s + \lfloor m^{1/\alpha} \rfloor )^{\alpha} \rfloor - m)^2 + (s + m^{1/\alpha} - m^{1/\alpha_1} )^2}.
\end{split}
\end{equation}
It's easy to see that the sum above converges and goes to $0$ if $m$ goes to infinity. Moreover, let's consider the sum
$$
S := c m^{3/(2 \alpha) - 1/2} \sum_{s = 1}^{\infty} \frac{ (s + m^{1/\alpha} +1)^{\alpha} - m}{( \lfloor (s + \lfloor m^{1/\alpha} \rfloor )^{\alpha} \rfloor - m)^2 + (s + m^{1/\alpha} - m^{1/\alpha_1} )^2}.
$$
Note that 
\begin{equation}
\begin{split}
& c m^{3/(2 \alpha) - 1/2} \sum_{s = 1}^{\infty} \frac{ (s + m^{1/\alpha} +1)^{\alpha} - m}{( \lfloor (s + \lfloor m^{1/\alpha} \rfloor )^{\alpha} \rfloor - m)^2 + (s + m^{1/\alpha} - m^{1/\alpha_1} )^2} \\
& \leq c m^{3/(2 \alpha) - 1/2} \sum_{s = 1}^{\infty} \frac{ (s + m^{1/\alpha} +1)^{\alpha} - m}{(s + m^{1/\alpha} - m^{1/\alpha_1} )^2}.
\end{split}
\end{equation}
For all $0 < \alpha < 1$, there is a $M > 0$ large enough such that for all $s > 0$ and $m' > M$,
$$
\frac{\partial}{\partial m} \Bigg( c m^{3/(2 \alpha) - 1/2} \sum_{s = 1}^{\infty} \frac{ (s + m^{1/\alpha} +1)^{\alpha} - m}{(s + m^{1/\alpha} - m^{1/\alpha_1} )^2} \Bigg) \Bigg|_{m = m'} < 0. 
$$
So the sum $S$ goes to $0$ if $m$ goes to infinity. Hence, we can take $n = \lfloor m^{3/(2 \alpha) - 1/2} \rfloor$. Note that $3/(2 \alpha) - 1/2 > 1/\alpha $. Then for any $y \in L_m $ with $|y^{(1)}| \leq \lfloor m^{1 / \alpha_1} \rfloor $, we have
$$
\lim_{m \rightarrow \infty} n \sum_{z \in \widetilde{A}_m} P_{(0,n)} (S_{\tau_{\widetilde{A}_m \cup L_m}} = z) P_z (S_{\tau_{L_m}} = y) = 0, 
$$
and
$$
\lim_{m \rightarrow \infty} \pi n P_{(0,n)} (S_{\tau_{A \cup L_m}} = y) = 1.
$$
Now fix $N > \max \{l, x_2 \}$. From the proof of Theorem 1 in \cite{procaccia2017stationary}, we know that the sequence $H_{A,j} (x)$ is decreasing for $j \geq N$. There exists a $M_2 > N$ such that for all $m > M_2$,  
$$
\Big| \pi n P_{(0,n)} (S_{\tau_{A \cup L_m}} = y) - 1 \Big| < \frac{\epsilon}{2 H_{A,N} (x)}.
$$
Therefore,
$$
\Bigg| \sum_{y \in L_{m} , |y^{(1)}| \leq \lfloor m^{1 / \alpha_1} \rfloor } \Big( \pi n P_{(0,n)} (S_{\tau_{A \cup L_m}} = y) -1 \Big) P_y (S_{\tau_{A}} = x) \Bigg| < \frac{\epsilon}{2}.
$$
Now take $m > \max \{M_1, M_2 \}$, and the proof is complete. 
\end{proof}

The following theorem is a direct consequence of Lemma \ref{l5} and Lemma \ref{l6}.

\begin{theorem}
Let $A \subset \upspace$ be an infinite set that has polynomial sub-linear growth. For any $x \in \upspace$,
$$
\asharm_{A} (x) := \lim_{ N \rightarrow \infty } \asharm_{A,N} (x)
$$
exists, and $\asharm_{A} (x) = \sharm_{A} (x)$.
\end{theorem}

\begin{proof}
Let $\epsilon > 0$. By Lemma \ref{l5} and Lemma \ref{l6}, there is an $M > 0$ such that for all $m > M$,
$$
|\sharm_{A,m} (x) - \asharm_{A, m} (x)| < \epsilon.
$$ 
We know
$$
\lim_{ m \rightarrow \infty } \sharm_{A,m} (x) = \sharm_{A} (x).
$$
Hence,
$$
\asharm_{A} (x) := \lim_{ m \rightarrow \infty } \asharm_{A,m} (x)
$$
exists and $\asharm_{A} (x) = \sharm_{A} (x)$.
\end{proof}

\section{Proof of the Main Theorem}
\label{s2}

In order to prove Theorem \ref{theorem part 2}, we first show its special situation when $A=L_0$, which can be stated as the following result on the asymptotic of regular harmonic measures: Let $D_n=[-n,n]\times\{0\}$ to be the horizontal line segment of interest. In this section we proved that  
\begin{theorem}
\label{theorem harmonic limit}
There is a constant $c\in (0,\infty)$ such that 
\beq
\label{harmonic limit}
\lim_{n\to\infty} n \harm_{D_n}(0)=c. 
\eeq
\end{theorem}

The structure of this section is as follows: In subsections \ref{sub2} and \ref{sub3} we outline the proof of Theorem \ref{theorem harmonic limit} and Theorem \ref{theorem part 2}. Then in the following subsections, we give the detailed proof of the required propositions and lemmas.

\subsection{Proof of Theorem \ref{theorem harmonic limit}}
\label{sub2}
Theorem \ref{theorem harmonic limit} can be proved according to the following outline: first, we show that $n \harm_{D_n}(0)$ has finite and positive upper and lower limits: 
\begin{Proposition}
\label{Proposition limsup}
There is a constant $C\in (0,\infty)$ such that 
\beq
\label{limsup}
\limsup_{n\to\infty} n \harm_{D_n}(0)\le C. 
\eeq
\end{Proposition}

\begin{Proposition}
\label{Proposition liminf}

There is a constant $c\in (0,\infty)$ such that 
\beq
\label{liminf}
\liminf_{n\to\infty} n \harm_{D_n}(0)\ge c. 
\eeq
\end{Proposition}

The two propositions above guarantee that the decaying rate of $\harm_{D_n}(0)$ is of order $1/n$. To show $\limsup=\liminf$, we further show the following coupling result:

\begin{Proposition}
\label{Proposition coupling}

For any $\ep>0$ there is a $\delta>0$ such that for all sufficiently large $n$ and any $x\in [-\delta n, \delta n]\times \{0\}$, we have 
\beq
\label{coupling}
\Big|\harm_{D_n}(0)-\harm_{D_n}(x)\Big|<\frac{\ep}{n}
\eeq
\end{Proposition}

Let $\bar B(0,R) = \{ x \in \mathbb{R}^2 : ||x||_2 < R \}$ be the continuous ball of radius $R$ in $\RR^2$. For standard Brownian motion $B(t)$ and subset $A\subset \RR^2$, define stopping time 
$$
T_A=\inf\{t\ge 0, \ B(t)\in A\}.
$$
For subset $A\subset \RR^2$, $\bharm_A$ denotes the continuous harmonic measure with respect to $A$.

\begin{lemma}
\label{brownian}
Fix $\delta \in (0,1)$, then
$$
\lim_{n\to\infty} \harm_{D_{n}}\left( [-\delta n, \delta n]\times \{0\}\right) = \bharm_{[-1,1] \times \{0\} } ( [-\delta, \delta]\times \{0\} ) .
$$

\end{lemma}

Once one has shown Proposition \ref{Proposition limsup}-\ref{brownian}, the proof of Theorem \ref{theorem harmonic limit} is mostly straightforward. Now suppose the limit in \eqref{harmonic limit} does not exist. Then by Proposition \ref{Proposition limsup} we must have 
\beq
\label{contradiction 1}
0<\liminf_{n\to\infty} n \harm_{D_n}(0)<\limsup_{n\to\infty} n \harm_{D_n}(0)<\infty. 
\eeq
Let 
$$
\ep_0=\frac{\limsup_{n\to\infty} n \harm_{D_n}(0)-\liminf_{n\to\infty} n \harm_{D_n}(0)}{5}>0. 
$$
By Proposition \ref{Proposition coupling}, we have there are $\delta_0>0$ and $N_0<\infty$ such that for all $n>N_0$ and any $x\in  [-\delta_0 n, \delta_0 n]\times \{0\}$
$$
\Big|\harm_{D_n}(0)-\harm_{D_n}(x)\Big|<\frac{\ep}{n}. 
$$
Moreover, for any $N>N_0$, there are $n_1,n_2>N$ such that 
$$
 n_1\harm_{D_{n_1}}(0)<\liminf_{n\to\infty} n \harm_{D_n}(0)+\ep_0
$$
and that 
$$
 n_2\harm_{D_{n_2}}(0)>\limsup_{n\to\infty} n \harm_{D_n}(0)-\ep_0. 
$$
At the same time, we have for the $\delta_0>0$ defined above, 
\beq
\label{contradiction 2}
\begin{aligned}
\harm_{D_{n_1}}\left( [-\delta_0 n_1, \delta_0 n_1]\times \{0\}\right)&=\sum_{x\in [-\delta_0 n_1, \delta_0 n_1]\times \{0\}}\harm_{D_{n_1}}(x)\\
&\le \frac{\lfloor\delta_0 n_1\rfloor+1}{n_1}\left[\liminf_{n\to\infty} n \harm_{D_n}(0)+2\ep_0 \right]
\end{aligned}
\eeq
and
\beq
\label{contradiction 3}
\begin{aligned}
\harm_{D_{n_2}}\left( [-\delta_0 n_2, \delta_0 n_2]\times \{0\}\right)&=\sum_{x\in [-\delta_0 n_2, \delta_0 n_2]\times \{0\}}\harm_{D_{n_2}}(x)\\
&\ge \frac{\lfloor\delta_0 n_2\rfloor+1}{n_2}\left[\limsup_{n\to\infty} n \harm_{D_n}(0)-2\ep_0 \right]. 
\end{aligned}
\eeq
But by Lemma \ref{brownian},
$$
\lim_{n\to\infty} \harm_{D_{n}}\left( [-\delta_0 n, \delta_0 n]\times \{0\}\right) = \bharm_{[-1,1] \times \{0\}} ( [-\delta_0, \delta_0]\times \{0\} ) ,
$$
which contradicts with \eqref{contradiction 2} and \eqref{contradiction 3}. \qed

\subsection{Proof of Theorem \ref{theorem part 2}}
\label{sub3}

Define $\alpha_1=(1+\alpha)/2\in (0,1)$ and $Box(n)=[-n,n]\times \left[0,\lfloor n^{\alpha_1}\rfloor\right]$. Recalling the definition of regular harmonic measure, and the fact that $A_n\subset Box(n)$ for all sufficiently large $n$, we have for any $x\in A\setminus L_0$, 
$$
\harm_{A_n}(x)=\sum_{y\in \partial^{in}Box(n)} \harm_{Box(n)}(y) P_y\left(S_{\bar \tau_{A_n}}=x\right). 
$$
Then define
$$
\begin{aligned}
&\partial^{in}_uBox(n)=[-n,n]\times\left\{ \lfloor n^{\alpha_1}\rfloor\right\}\\
&\partial^{in}_dBox(n)=[-n,n]\times\left\{0\right\}\\
&\partial^{in}_lBox(n)=\{-n\},\times\left[1, \lfloor n^{\alpha_1}\rfloor-1\right]\\
&\partial^{in}_rBox(n)=\{n\},\times\left[1, \lfloor n^{\alpha_1}\rfloor-1\right]
\end{aligned}
$$
to be the four edges of $\partial^{in}Box(n)$. Noting that $L_0 \subset A$, it is easy to see that for any $y\in \partial^{in}_dBox(n)=[-n,n]\times\left\{0\right\}$, $P_y\left(S_{\bar \tau_{A_n}}=x\right)=0$. Moreover, define $\alpha_2=(7+\alpha)/8$, and 
$$
l_n=\left[- \lfloor n^{\alpha_2}\rfloor, \lfloor n^{\alpha_2}\rfloor\right]\times\left\{ \lfloor n^{\alpha_1}\rfloor\right\}
$$
to be the middle section of $\partial^{in}_uBox(n)$ and denote $l_n^c=\partial^{in}_lBox(n)\cup\partial^{in}_rBox(n)\cup\partial^{in}_uBox(n)\setminus l_n$. We further have the decomposition as follows: 
\beq
\label{part 2 1}
\begin{aligned}
\harm_{A_n}(x)&=\sum_{y\in l^c_n} \harm_{Box(n)}(y) P_y\left(S_{\bar \tau_{A_n}}=x\right)+\sum_{y\in l_n} \harm_{Box(n)}(y) P_y\left(S_{\bar \tau_{A_n}}=x\right).
\end{aligned}
\eeq
From \eqref{part 2 1}, we first note that $\harm_{Box(n)}(y)$ sums up to 1, which implies that 
\beq
\sum_{y\in l^c_n} \harm_{Box(n)}(y) P_y\left(S_{\bar \tau_{A_n}}=x\right) \le  \max_{y\in l^c_n} P_y\left(S_{\bar \tau_{A_n}}=x\right).
\eeq
Thus our first step is to prove 
\begin{Proposition}
\label{lemma away}
For $Box(n)$, $l_n$ and $l_n^c$ defined as above, we have 
\beq
\label{away_0}
\lim_{n\to\infty} n \cdot \max_{y\in l^c_n} P_y\left(S{\bar \tau_{A_n}}=x\right)=0. 
\eeq
\end{Proposition}

With Proposition \ref{lemma away}, it sufficient for us to concentrate on the asymptotic of 
$$
\sum_{y\in l_n} \harm_{Box(n)}(y) P_y\left(S_{\bar \tau_{A_n}}=x\right).
$$
We are to show that 
\begin{Proposition}
\label{lemma stationary harmonic}
For any $x\in A$ and the truncations $A_n$ defined in \eqref{truncation}
\beq
\label{stationary harmonic}
\lim_{n\to\infty} \sum_{y\in l_n}  P_y\left(S_{\bar \tau_{A_n}}=x\right)= \sharm_A(x). 
\eeq
\end{Proposition}
\noindent and that 
\begin{Proposition}
\label{lemma coupling two}
For any $\ep>0$, there is a $N_0<\infty$ such that for all $n\ge N_0$ and all $y\in l_n$, 
\beq
\label{coupling two 1}
\left|2\harm_{Box(n)}(y)-\harm_{D_n}(0)\right|<\ep/n. 
\eeq
\end{Proposition}
Once we have proved the lemmas above, Theorem \ref{theorem part 2} follows immediately from the combination of Proposition \ref{lemma away}- \ref{lemma coupling two}, together with Theorem \ref{theorem harmonic limit}. \qed

\subsection{Existence of upper and lower limit}
\subsubsection{Bounds between harmonic measure and escaping probability}
In this subsection we prove Proposition \ref{Proposition limsup} and \ref{Proposition liminf}. First, recalling the notation 
$$
\harm_D(y,x)= P_y(\tau_D=\tau_x),
$$ 
with standard time reversibility argument, see Lemma 2 of \cite{kesten1987hitting}, we have for any $n$ and $x\in D_n$
$$
\begin{aligned}
\harm_{D_n}(x)&=\lim_{R\to\infty} \frac{1}{\left| \partial^{out}B(0,R)\right|}\sum_{y\in \partial^{out}B(0,R)} \harm_{D_n}(y,x)\\
&=\lim_{R\to\infty} \frac{1}{\left| \partial^{out}B(0,R)\right|}E_x\left[\text{number of visits to $\partial^{out}B(0,R)$ in } [0,\tau_{D_n}) \right]. 
\end{aligned}
$$
Note that there is a finite constant $C$ independent to $R$ such that 
$$
\frac{1}{\left| \partial^{out}B(0,R)\right|}\le \frac{C}{R}. 
$$
At the same time, define $C_n=\left[-\lfloor n/2\rfloor,0\right]\times \{0\}\subset D_n$ and apply Lemma 3-4 of \cite{kesten1987hitting} with $r=n$, 
$$
\begin{aligned}
&E_x\left[\text{number of visits to $\partial^{out}B(0,R)$ in } [0,\tau_{D_n}) \right]\\
&\le \frac{P_x(\tau_{R}<\tau_{D_n})}{\min_{w\in \partial^{out}B(0,R)} P_w\left(\tau_{D_n}<\tau_{R} \right)}\\
&\le C R\log(R) P_x(\tau_{R}<\tau_{D_n})\\
&=C R\log(R)\left( \sum_{z\in \partial^{out}B(0,2n)} P_x\left(\tau_{2n}<\tau_{D_n}, S_{\tau_{2n}}=z\right) P_z(\tau_{R}<\tau_{D_n})\right)\\
&\le C R\log(R)\left( \sum_{z\in \partial^{out}B(0,2n)} P_x\left(\tau_{2n}<\tau_{D_n}, S_{\tau_{2n}}=z\right) P_z(\tau_{R}<\tau_{C_n})\right)\\
&\le C R\log(R)P_x\left(\tau_{2n}<\tau_{D_n}\right)\max_{z\in \partial^{out}B(0,2n)} P_z(\tau_{R}<\tau_{C_n})\\
&\le C R  P_x\left(\tau_{2n}<\tau_{D_n}\right). 
\end{aligned}
$$ 
Thus, there is a finite constant $C$ independent to $n$ such that 
\beq
\label{escape upper}
\harm_{D_n}(x)\le C  P_x\left(\tau_{2n}<\tau_{D_n}\right). 
\eeq
On the other hand, by Lemma 3.2 of \cite{procaccia2017stationary}, there is a constant $C<\infty$ independent to the choice of $n$ and $R\gg n$ such that for all $w\in \partial^{out}B(0,R)$ 
\beq
\label{lemma 3.2}
\BP_w(\tau_{D_n}<\tau_R)\le C[R\log(R)]^{-1}. 
\eeq
Thus 
$$
\begin{aligned}
&E_x\left[\text{number of visits to $\partial^{out}B(0,R)$ in } [0,\tau_{D_n}) \right]\\
&\ge \frac{P_x(\tau_{R}<\tau_{D_n})}{\max_{w\in \partial^{out}B(0,R)} P_w\left(\tau_{D_n}<\tau_{R} \right)}\\
&\ge c R\log(R) P_x(\tau_{R}<\tau_{D_n}).
\end{aligned}
$$ 
At the same time, by Lemma 3.3 of \cite{procaccia2017stationary}, there are constants $2< c_0 <\infty$ and $c>0$ independent to the choice of $n$ and $R\gg n$ such that for any $z\in \partial^{out}B(0,c_0 n)$
\beq
\label{lemma 3.3}
P_z(\tau_{R}<\tau_{D_n})\ge \frac{c}{\log(R)}. 
\eeq
Thus we have 
\begin{align*}
P_x(\tau_{R}<\tau_{D_n})&=\sum_{z\in \partial^{out}B(0,c_0 n)} P_x\left(\tau_{c_0 n}<\tau_{D_n}, S_{\tau_{c_0 n}}=z\right) P_z(\tau_{R}<\tau_{D_n})\\
&\ge c R P_x\left(\tau_{c_0 n}<\tau_{D_n}\right). 
\end{align*}
which implies that 
\beq
\label{escape lower}
\harm_{D_n}(x)\ge c  P_x\left(\tau_{c_0 n}<\tau_{D_n}\right). 
\eeq

\subsubsection{Proof of Proposition  \ref{Proposition limsup}} 
With Lemma \ref{lemma combinatorial} and recalling the fact that $I_n\subset B(0,2n)$, we have that 
\beq
\label{escape upper 1}
\begin{aligned}
\BP_0(\tau_{2n}<\tau_{D_n})&\le \BP_0(\tau_{I_n}<\tau_{D_n})\\
&=P_0\left(S_{T_n}\in L_n\cup \partial^{in}_{r} I_n\cup \partial^{in}_{u} I_n \right)\\
&\le 2P_0\left(S_{T_n}\in \partial^{in}_{u} I_n \right).
\end{aligned}
\eeq
Moreover, note that 
\beq
\label{escape upper 2}
P_0\left(S_{T_n}\in \partial^{in}_{u} I_n \right) \leq P_0\left(\tau_{L_n}<\tau_{L_0}\right)=\frac{1}{4n}. 
\eeq
Thus by  \eqref{escape upper}, \eqref{escape upper 1} and \eqref{escape upper 2}, the proof of Proposition  \ref{Proposition limsup} is complete. \qed

\subsubsection{Proof of Proposition  \ref{Proposition liminf}}

With \eqref{escape lower}, in order to Proposition  \ref{Proposition liminf}, it is sufficient to show that 
\begin{lemma}
\label{lemma lower}
For any $k\ge 2$, there is a $c_k>0$ such that 
$$
 \BP_0\left(\tau_{kn}<\tau_{D_n}\right)\ge \frac{c_k}{n}. 
$$
\end{lemma}
\begin{proof}
Note that for a simple random walk starting from $0$, it is easy to see that 
$$
\tau_{kn}\le \tau_{L_{kn}}, \ \tau_{L_0}\le \tau_{D_n}.
$$
Thus we have 
$$
 \BP_0\left(\tau_{kn}<\tau_{D_n}\right)\ge  \BP_0\left(\tau_{L_{kn}}<\tau_{L_0} \right)=\frac{1}{4kn}
$$
and the proof of this lemma is complete. 
\end{proof}
With Lemma \ref{lemma lower}, the proof of Proposition  \ref{Proposition liminf} is complete. \qed

\subsection{Proof of Proposition  \ref{Proposition coupling}}

For the proof of Proposition \ref{Proposition coupling}, we without loss of generality assume that the first coordinate of $x$ is an even number, see Remark \ref{remark even} for details. With Proposition \ref{Proposition limsup}  and \ref{Proposition liminf}, by spatial translation it is easy to see there are constants $0<c<C<\infty$ such that for all $x\in [-n/2, n/2]$ 
\beq
\label{bound_x}
\frac{c}{n}<\harm_{D_n}(x)<\frac{C}{n}. 
\eeq
Moreover, recall that 
$$
\begin{aligned}
\harm_{D_n}(x)&=\lim_{R\to\infty} \frac{1}{\left| \partial^{out}B(0,R)\right|}\sum_{y\in \partial^{out}B(0,R)} \harm_{D_n}(y,x)\\
&=\lim_{R\to\infty} \frac{1}{\left| \partial^{out}B(0,R)\right|}E_x\left[\text{number of visits to $\partial^{out}B(0,R)$ in } [0,\tau_{D_n}) \right]. 
\end{aligned}
$$
Thus for any $n$ and $x$, there has to be a $R_0$ such that for all $R\ge R_0$,
$$
\left| \harm_{D_n}(x)- \frac{1}{\left| \partial^{out}B(0,R)\right|}E_x\left[\text{number of visits to $\partial^{out}B(0,R)$ in } [0,\tau_{D_n}) \right] \right|<\frac{\ep}{4n}
$$
and 
$$
\left| \harm_{D_n}(0)- \frac{1}{\left| \partial^{out}B(0,R)\right|}E_0\left[\text{number of visits to $\partial^{out}B(0,R)$ in } [0,\tau_{D_n}) \right] \right|<\frac{\ep}{4n}.
$$
At the same time 
$$
\begin{aligned}
&E_x\left[\text{number of visits to $\partial^{out}B(0,R)$ in } [0,\tau_{D_n}) \right] \\
=&\sum_{z\in \partial^{out}B(0,2n)}\hspace{-0.2 in}\BP_{x} \left(\tau_{2n}<\tau_{D_n}, S_{\tau_{2n}}=z\right)\hspace{-0.2 in}\sum_{w\in \partial^{out}B(0,R)}\hspace{-0.2 in}\frac{\BP_z\left(\tau_{R}<\tau_{D_n}, S_{\tau_{R}}=w\right)}{\BP_w\left(\tau_{D_n}<\tau_R\right)}
\end{aligned}
$$
and 
$$
\begin{aligned}
&E_0\left[\text{number of visits to $\partial^{out}B(0,R)$ in } [0,\tau_{D_n}) \right] \\
=&\sum_{z\in \partial^{out}B(0,2n)}\hspace{-0.2 in}\BP_{0} \left(\tau_{2n}<\tau_{D_n}, S_{\tau_{2n}}=z\right)\hspace{-0.2 in}\sum_{w\in \partial^{out}B(0,R)}\hspace{-0.2 in}\frac{\BP_z\left(\tau_{R}<\tau_{D_n}, S_{\tau_{R}}=w\right)}{\BP_w\left(\tau_{D_n}<\tau_R\right)}.
\end{aligned}
$$
Thus we have 
\beq
\label{coupling 1}
\begin{aligned}
&\left|\harm_{D_n}(x)-\harm_{D_n}(0)\right|\\
&\le \frac{1}{\left| \partial^{out}B(0,R)\right|}\sum_{z\in \partial^{out}B(0,2n)}\hspace{-0.2 in}\left|\BP_{0} \left(\tau_{2n}<\tau_{D_n}, S_{\tau_{2n}}=z\right)-\BP_{x} \left(\tau_{2n}<\tau_{D_n}, S_{\tau_{2n}}=z\right)\right|\\
&\hspace{ 1.4 in} \cdot \left(\sum_{w\in \partial^{out}B(0,R)}\hspace{-0.2 in}\frac{\BP_z\left(\tau_{R}<\tau_{D_n}, S_{\tau_{R}}=w\right)}{\BP_w\left(\tau_{D_n}<\tau_R\right)} \right)+\frac{\ep}{2n}. 
\end{aligned}
\eeq
Again by Lemma 3-4 of \cite{kesten1987hitting} with $r=n$, we have there is a constant $C<\infty$ such that for all $n$, $R\gg n$ and $z\in \partial^{out}B(0,2n)$
\beq
\label{coupling 2}
\begin{aligned}
&\frac{1}{\left| \partial^{out}B(0,R)\right|}\left(\sum_{w\in \partial^{out}B(0,R)}\hspace{-0.2 in}\frac{\BP_z\left(\tau_{R}<\tau_{D_n}, S_{\tau_{R}}=w\right)}{\BP_w\left(\tau_{D_n}<\tau_R\right)} \right)\\
&\le \frac{\BP_z\left(\tau_{R}<\tau_{D_n}\right)}{\left| \partial^{out}B(0,R)\right|\min_{w\in \partial^{out}B(0,R)}\BP_w\left(\tau_{D_n}<\tau_{R} \right)}\le C. 
\end{aligned}
\eeq
Thus by \eqref{coupling 1} and  \eqref{coupling 2}, in order to prove Proposition \ref{coupling}, it suffices to show the following lemma:
\begin{lemma}
\label{lemma coupling 1}
For any $\ep>0$ there is a $\delta>0$ such that for all sufficiently large $n$ and any $x\in [-\delta n, \delta n]\times \{0\}$, we have 
\beq
\label{coupling 3}
\sum_{z\in \partial^{out}B(0,2n)}\hspace{-0.2 in}\left|\BP_{0} \left(\tau_{2n}<\tau_{D_n}, S_{\tau_{2n}}=z\right)-\BP_{x} \left(\tau_{2n}<\tau_{D_n}, S_{\tau_{2n}}=z\right)\right|<\frac{\ep}{n}. 
\eeq
\end{lemma}
\begin{proof}
For any $\ep>0$, define $\delta=e^{-\ep^{-1}}>0$. In order to prove this lemma, we construct the following coupling between the simple random walk starting from 0 and $x\in [-\delta n, \delta n]\times \{0\}$:
\begin{enumerate}[(i)]
\item Define subset $A^\ep_n=[-\lfloor n/2\rfloor,\lfloor n/2\rfloor]\times [0,\lfloor \ep n\rfloor]$. 
\item Let $\{\bar S_k\}_{k=0}^\infty$ be a simple random walk starting from 0, $\bar T_n^\ep=\inf\{k: \ \bar S_k\in \partial^{in} A^\ep_n\}$, and $x_{n}^\ep=\bar S_{\bar T_n^\ep}$.  
\item For $k\le \bar T_n^\ep$, let $S_{1,k}=\bar S_k$ and $S_{2,k}=\bar S_k+x$. 
\item Let $\left\{\hat S_{1,k}\right\}_{k=0}^\infty$ and $\left\{\hat S_{2,k}\right\}_{k=0}^\infty$ be two simple random walks starting from $x_{n}^\ep$ and $x_{n}^\ep+x$ and coupled under the maximal coupling.
\item For $k> \bar T_n^\ep$, let $S_{1,k}=\hat S_{1,k- T_n^\ep}$ and $S_{2,k}=\hat S_{2,k- T_n^\ep}$.
\end{enumerate}
\begin{remark}
\label{remark even}
In Step (iv) we use the assumption that the first coordinate of $x$ is an even number. Otherwise, one can construct $\hat S_{1,k}$ starting from $x_{n}^\ep$ and $\hat S_{2,k}$ starting uniformly from $B(x_{n}^\ep+x,1)$ under maximal coupling. 
\end{remark}
According to strong Markov property, it is easy to see that $S_{1,k}$ and $S_{2,k}$ form two simple random walks starting from $0$ and $x$. Let $\tau^{(1)}_{\cdot}$ and  $\tau^{(2)}_{\cdot}$ be the stopping time with respect to $S_{1,k}$ and $S_{2,k}$ respectively. Thus
$$
\begin{aligned}
&\sum_{z\in \partial^{out}B(0,2n)}\hspace{-0.2 in}\left|\BP_{0} \left(\tau_{2n}<\tau_{D_n}, S_{\tau_{2n}}=z\right)-\BP_{x} \left(\tau_{2n}<\tau_{D_n}, S_{\tau_{2n}}=z\right)\right|\\
= &\sum_{z\in \partial^{out}B(0,2n)}\hspace{-0.2 in}\left|\BP_{0} \left(\tau^{(1)}_{2n}<\tau^{(1)}_{D_n}, S_{1,\tau^{(1)}_{2n}}=z\right)-\BP_{x} \left(\tau^{(2)}_{2n}<\tau^{(2)}_{D_n}, S_{2,\tau^{(2)}_{2n}}=z\right)\right|. 
\end{aligned}
$$
Again we introduce 
$$
U^\ep_n=[-\lfloor n/2\rfloor,\lfloor n/2\rfloor]\times \lfloor \ep n\rfloor, \ \  B^\ep_n=[-\lfloor n/2\rfloor,\lfloor n/2\rfloor]\times0
$$
and
$$
L^\ep_n=-\lfloor n/2\rfloor\times [1,\lfloor \ep n\rfloor-1] \ \ R^\ep_n=\lfloor n/2\rfloor\times [1,\lfloor \ep n\rfloor-1]
$$
as the four edges of $\partial^{in} A^\ep_n$. Note that for all $\ep<1/3$
$$
\begin{aligned}
&\left\{\tau^{(1)}_{2n}<\tau^{(1)}_{D_n}\right\}\cap \left\{\bar S_{\bar T_n^\ep}\in B^\ep_n \right\}=\emptyset, \ 
&\left\{\tau^{(2)}_{2n}<\tau^{(2)}_{D_n}\right\}\cap \left\{\bar S_{\bar T_n^\ep}\in B^\ep_n \right\}=\emptyset. 
\end{aligned}
$$
Thus for any $z\in \partial^{out}B(0,2n)$, we have 
$$
\begin{aligned}
\BP_{0} \left(\tau^{(1)}_{2n}<\tau^{(1)}_{D_n}, S_{1,\tau^{(1)}_{2n}}=z\right)=&\BP_{0} \left(\bar S_{\bar T_n^\ep}\in U^\ep_n, \ \tau^{(1)}_{2n}<\tau^{(1)}_{D_n}, \ S_{1,\tau^{(1)}_{2n}}=z\right)\\
+&\BP_{0} \left(\bar S_{\bar T_n^\ep}\in L^\ep_n\cup R^\ep_n, \ \tau^{(1)}_{2n}<\tau^{(1)}_{D_n}, \ S_{1,\tau^{(1)}_{2n}}=z\right)
\end{aligned}
$$
and 
$$
\begin{aligned}
\BP_{x} \left(\tau^{(2)}_{2n}<\tau^{(2)}_{D_n}, S_{2,\tau^{(2)}_{2n}}=z\right)=&\BP_{x} \left(\bar S_{\bar T_n^\ep}\in U^\ep_n, \ \tau^{(2)}_{2n}<\tau^{(2)}_{D_n}, \ S_{2,\tau^{(2)}_{2n}}=z\right)\\
+&\BP_{x} \left(\bar S_{\bar T_n^\ep}\in L^\ep_n\cup R^\ep_n, \ \tau^{(2)}_{2n}<\tau^{(2)}_{D_n}, \ S_{2,\tau^{(2)}_{2n}}=z\right). 
\end{aligned}
$$
Thus we have 
\beq
\label{coupling 4}
\begin{aligned}
&\sum_{z\in \partial^{out}B(0,2n)}\hspace{-0.2 in}\left|\BP_{0} \left(\tau^{(1)}_{2n}<\tau^{(1)}_{D_n}, S_{1,\tau^{(1)}_{2n}}=z\right)-\BP_{x} \left(\tau^{(2)}_{2n}<\tau^{(2)}_{D_n}, S_{2,\tau^{(2)}_{2n}}=z\right)\right|\\
\le &\sum_{z\in \partial^{out}B(0,2n)}\hspace{-0.2 in}\left|\BP\left(\bar S_{\bar T_n^\ep}\in U^\ep_n, \ \tau^{(1)}_{2n}<\tau^{(1)}_{D_n}, \ S_{1,\tau^{(1)}_{2n}}=z\right)- \BP \left(\bar S_{\bar T_n^\ep}\in U^\ep_n, \ \tau^{(2)}_{2n}<\tau^{(2)}_{D_n}, \ S_{2,\tau^{(2)}_{2n}}=z\right)\right|\\
+&\sum_{z\in \partial^{out}B(0,2n)}\hspace{-0.2 in}\BP \left(\bar S_{\bar T_n^\ep}\in L^\ep_n\cup R^\ep_n, \ \tau^{(1)}_{2n}<\tau^{(1)}_{D_n}, \ S_{1,\tau^{(1)}_{2n}}=z\right) \\
+&\sum_{z\in \partial^{out}B(0,2n)}\hspace{-0.2 in}\BP \left(\bar S_{\bar T_n^\ep}\in L^\ep_n\cup R^\ep_n, \ \tau^{(2)}_{2n}<\tau^{(2)}_{D_n}, \ S_{2,\tau^{(2)}_{2n}}=z\right)\\
\le & \sum_{z\in \partial^{out}B(0,2n)}\hspace{-0.2 in}\left|\BP \left(\bar S_{\bar T_n^\ep}\in U^\ep_n, \ \tau^{(1)}_{2n}<\tau^{(1)}_{D_n}, \ S_{1,\tau^{(1)}_{2n}}=z\right)- \BP \left(\bar S_{\bar T_n^\ep}\in U^\ep_n, \ \tau^{(2)}_{2n}<\tau^{(2)}_{D_n}, \ S_{2,\tau^{(2)}_{2n}}=z\right)\right|\\
+ & 2\BP \left(\bar S_{\bar T_n^\ep}\in L^\ep_n\cup R^\ep_n\right). 
\end{aligned}
\eeq
In order to control the right hand side of \eqref{coupling 4}, we first concentrate on controlling its second term. Note that by invariance principle it is easy to check that there is a constant $c>0$ such that for any integer $m>1$ and any integer $j$ with $|j|\le m$, we have
\beq
\label{invariance upper 1}
\BP_{(0,j)}\left(\tau_{\partial^{in}_{l} I_m \cup \partial^{in}_{r} I_m}<\tau_{\partial^{in}_{u} I_m \cup \partial^{in}_{b} I_m}\right)<1-c. 
\eeq
Moreover, by Lemma \ref{lemma combinatorial}, 
\begin{equation}
\label{ded}
P_{(0,0)} (\tau_{\partial^{in}_{l} I_m \cup \partial^{in}_{r} I_m}<\tau_{\partial^{in}_{u} I_m \cup \partial^{in}_{b} I_m} ) \leq P_{(0,0)} (\tau_{L_m} < \tau_{L_0}) = \frac{1}{4\epsilon m}.
\end{equation}
In the rest of the proof we call the event in \eqref{invariance upper 1} a side escaping event. The detailed proof of \eqref{invariance upper 1} follows exactly the same argument as the proof of Equation (11) in \cite{procaccia2017sets}, which can also be illustrated in the following figure: 

\begin{figure}[h]
\label{fig: invariance upper 1}
\centering 
\begin{tikzpicture}[scale=0.5]
\tikzstyle{redcirc}=[circle,
draw=black,fill=red,thin,inner sep=0pt,minimum size=2mm]
\tikzstyle{bluecirc}=[circle,
draw=black,fill=blue,thin,inner sep=0pt,minimum size=2mm]

\draw [black,thin] (5,0) to (0,0);

\draw[black,thin] (5,0) to (10,0);

\draw [black,thin] (0,0) to (0,5);

\draw [black,thin] (10,0) to (10,5);

\draw [black,thin] (0,5) to (10,5);

\draw [black,dashed] (5,0) to (5,5);

\draw [red,dashed] (0,5) to (0,8) ;

\draw [red,dashed] (10,5) to (10,8) ;

\draw [red,dashed] (0,0) to (0,-2) ;

\draw [red,dashed] (10,0) to (10,-2) ;

\draw [red,dashed] (0,-2) to (10,-2) ;

\draw [red,dashed] (0,8) to (10,8) ;

\draw [red,dashed] (5,5) to (5,8) ;

\draw [red,dashed] (5,0) to (5,-2) ;

\node (v2) at (5,3) [bluecirc] {};
\node (v3) at (6,3) {(0,j)};

\node (v4) at (11,0) {(m,0)};

\node (v5) at (-1.1,0) {(-m,0)};

\node (v6) at (11.1,5) {(m,m)};

\node (v7) at (-1.2,5) {(-m,m)};

\node (v4) at (11.4,-2) {(m,j-m)};

\node (v5) at (-1.6,-2) {(-m,j-m)};

\node (v6) at (11.6,8) {(m,j+m)};

\node (v7) at (-1.7,8) {(-m,j+m)};

\draw [blue,thick] plot [smooth, tension=2.5] coordinates {(5,3) (4, 2) (6,1) (5.5, 4) (6, 2) (6.5, 3) (7.5, 4) (8.5, 3) (10, 2.5)};

\node (v2) at (10,2.5) [bluecirc] {};

\end{tikzpicture}
\caption{invariance principle for \eqref{invariance upper 1}}
\end{figure}
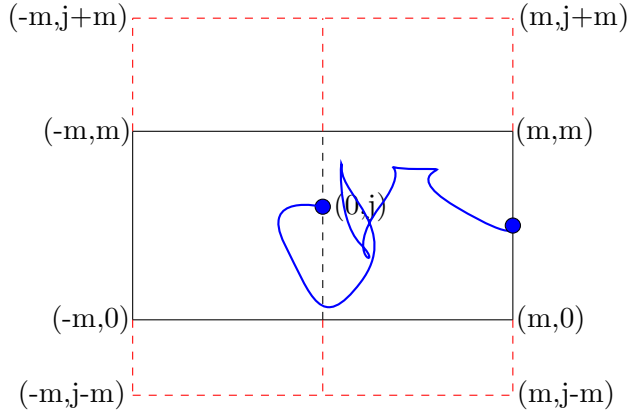

Moreover, define $m(\ep,n)=\lfloor \ep n\rfloor$. Note that in the event $\{\bar S_{\bar T_n^\ep}\in L^\ep_n\cup R^\ep_n\}$, our simple random walk has to first escape $A_{m(\ep,n)}$ through $L_{m(\ep,n)}\cup R_{m(\ep,n)}$ and then has at least $K(\ep,n)=\big\lfloor\lfloor n/2\rfloor/m(\ep,n)\big\rfloor$ independent times of side escaping events. Thus by Lemma \ref{lemma combinatorial}, \eqref{invariance upper 1}, (\ref{ded}), and the fact that  for all sufficiently small $\ep>0$, 
$$
K(\ep,n)=\big\lfloor\lfloor n/2\rfloor/m(\ep,n)\big\rfloor\ge \frac{1}{3\ep}
$$
we have 
\beq
\label{coupling 5}
\BP \left(\bar S_{\bar T_n^\ep}\in L^\ep_n\cup R^\ep_n\right)\le \frac{1}{4 \ep n} (1-c)^{\frac{1}{3\ep} -1}\ll \frac{\ep}{n}
\eeq
for all sufficiently small $\ep>0$. Thus in order to prove Lemma \ref{lemma coupling 1}, it suffices to show that 

\beq
\label{coupling 6}
\begin{aligned}
& \sum_{z\in \partial^{out}B(0,2n)}\hspace{-0.2 in}\left| P \left(\bar S_{\bar T_n^\ep}\in U^\ep_n, \ \tau^{(1)}_{2n}<\tau^{(1)}_{D_n}, \ S_{1,\tau^{(1)}_{2n}}=z\right)- P \left(\bar S_{\bar T_n^\ep}\in U^\ep_n, \ \tau^{(2)}_{2n}<\tau^{(2)}_{D_n}, \ S_{2,\tau^{(2)}_{2n}}=z\right)\right|\\
 & \ll \frac{\ep}{n}. 
 \end{aligned}
\eeq
Recall that in our construction, $\left\{\hat S_{1,k}\right\}_{k=0}^\infty$ and $\left\{\hat S_{2,k}\right\}_{k=0}^\infty$ are simple random walks coupled under the maximal coupling. Define events: 
$$
\mathcal{A}_1=\left\{\hat S_{1,k}\notin D_n\cup \partial^{out}B(0,2n), \ \forall k\le \ep^4 n^2 \right\}, 
$$
$$
\mathcal{A}_2=\left\{\hat S_{2,k}\notin D_n\cup \partial^{out}B(0,2n), \ \forall k\le \ep^4 n^2 \right\},
$$
and
$$
\mathcal{A}_3=\left\{\text{there exists a } k\le \ep^4 n^2 \text{ such that }  \hat S_{1,j}=\hat S_{1,j}, \forall j\ge k\right\}.
$$
By definition, one can easily see that 
\beq
\label{coupling 7}
\begin{aligned}
&\left\{\bar S_{\bar T_n^\ep}\in U^\ep_n, \ \tau^{(1)}_{2n}<\tau^{(1)}_{D_n}, \ S_{1,\tau^{(1)}_{2n}}=z\right\}\cap \mathcal{A}_1\cap \mathcal{A}_2\cap \mathcal{A}_3\\
=&\left\{\bar S_{\bar T_n^\ep}\in U^\ep_n, \ \tau^{(2)}_{2n}<\tau^{(2)}_{D_n}, \ S_{2,\tau^{(2)}_{2n}}=z\right\}\cap \mathcal{A}_1\cap \mathcal{A}_2\cap \mathcal{A}_3
\end{aligned}
\eeq
which implies that 
\beq
\label{coupling 8}
\begin{aligned}
& \sum_{z\in \partial^{out}B(0,2n)}\hspace{-0.2 in}\left|\BP\left(\bar S_{\bar T_n^\ep}\in U^\ep_n, \ \tau^{(1)}_{2n}<\tau^{(1)}_{D_n}, \ S_{1,\tau^{(1)}_{2n}}=z\right)- \BP \left(\bar S_{\bar T_n^\ep}\in U^\ep_n, \ \tau^{(2)}_{2n}<\tau^{(2)}_{D_n}, \ S_{2,\tau^{(2)}_{2n}}=z\right)\right|\\
 & \le 2\BP\left(\{\bar S_{\bar T_n^\ep}\in U^\ep_n\}\cap \mathcal{A}_1^c \right)+2\BP\left(\{\bar S_{\bar T_n^\ep}\in U^\ep_n\}\cap \mathcal{A}_2^c \right)+2\BP\left(\{\bar S_{\bar T_n^\ep}\in U^\ep_n\}\cap \mathcal{A}_3^c \right). 
 \end{aligned}
\eeq
Thus, it suffices to control the probabilities on the right hand side of \eqref{coupling 8}. For its first term, we have by Proposition 2.1.2 of \cite{lawler2010random} there are constants $c, \beta\in (0,\infty)$, independent to $n$ such that 
$$
\BP\left( \mathcal{A}_1^c \right)\le c e^{-\beta/\ep^2}, \ \ \BP\left( \mathcal{A}_2^c \right)\le c e^{-\beta/\ep^2}.
$$
And by strong Markov property we have
\beq
\label{coupling 9}
\BP\left(\{\bar S_{\bar T_n^\ep}\in U^\ep_n\}\cap \mathcal{A}_1^c \right)\le \frac{c e^{-\beta/\ep^2}}{\ep} n^{-1}\ll \frac{\ep}{n}
\eeq
and 
\beq
\label{coupling 10}
\BP\left(\{\bar S_{\bar T_n^\ep}\in U^\ep_n\}\cap \mathcal{A}_2^c \right)\le \frac{c e^{-\beta/\ep^2}}{\ep} n^{-1}\ll \frac{\ep}{n} 
\eeq
for all sufficiently small $\ep>0$. Finally, for the last term 
$$
\BP\left(\{\bar S_{\bar T_n^\ep}\in U^\ep_n\}\cap \mathcal{A}_3^c \right)
$$
recall that the first coordinate of $x$ is even and that $\left\{\hat S_{1,k}\right\}_{k=0}^\infty$ and $\left\{\hat S_{2,k}\right\}_{k=0}^\infty$ be two simple random walks starting from $x_{n}^\ep$ and $x_{n}^\ep+x$ and coupled under the maximal coupling. We have that 
$$
\BP\left( \mathcal{A}_3^c \right)\le d_{TV}\left(\hat S_{1,\lfloor \ep^4 n^2 \rfloor}, \hat S_{2,\lfloor \ep^4 n^2 \rfloor}\right)
$$
where $d_{TV}(\cdot,\cdot)$ stands for the total variation distance between the distributions of two random variables. On the other hand, note that 
$$
\begin{aligned}
d_{TV}\left(\hat S_{1,\lfloor \ep^4 n^2 \rfloor}, \hat S_{2,\lfloor \ep^4 n^2 \rfloor}\right)&= \frac{1}{2}\sum_{z\in \ZZ^2}\left|\BP\left(\hat S_{1,\lfloor \ep^4 n^2 \rfloor}=z\right)-\BP\left(\hat S_{2,\lfloor \ep^4 n^2 \rfloor}=z\right) \right|\\
&\le \frac{1}{2} \Bigg[ \BP\left(\hat S_{1,\lfloor \ep^4 n^2 \rfloor}\in B^c(0,2n)\right)+\BP\left(\hat S_{2,\lfloor \ep^4 n^2 \rfloor}\in B^c(0,2n)\right)\\
&+\sum_{z\in B(0,2n)}\left|\BP\left(\hat S_{1,\lfloor \ep^4 n^2 \rfloor}=z\right)-\BP\left(\hat S_{2,\lfloor \ep^4 n^2 \rfloor}=z\right) \right| \Bigg] . 
\end{aligned}
$$
And again by Proposition 2.1.2 of \cite{lawler2010random} there are constants $c, \beta\in (0,\infty)$, independent to $n$ such that
\beq
\label{total variation 1}
\BP\left(\hat S_{1,\lfloor \ep^4 n^2 \rfloor}\in B^c(0,2n)\right)\le  c e^{-\beta/\ep^4}, \ \BP\left(\hat S_{2,\lfloor \ep^4 n^2 \rfloor}\in B^c(0,2n)\right)\le c e^{-\beta/\ep^4}. 
\eeq
And for any $z\in B(0,2n)$, condition on  $\bar S_{\bar T_n^\ep}=x_{n}^\ep$, applying Proposition 4.1 of \cite{delmotte1999parabolic} with $x_0=x_{n}^\ep$, $n_0=\lfloor \ep^4 n^2 \rfloor$ and $R =\lfloor \ep^4 n\rfloor$, there are constant $h>0$ and $C<\infty$ independent to $n$ and the choice of $x_{n}^\ep$, 
$$
\begin{aligned}
&\left|\BP\left(\hat S_{1,\lfloor \ep^4 n^2 \rfloor}=z\Big|\bar S_{\bar T_n^\ep}=x_{n}^\ep\right)-\BP\left(\hat S_{2,\lfloor \ep^4 n^2 \rfloor}=z\Big| \bar S_{\bar T_n^\ep}=x_{n}^\ep\right) \right|\\
&\le C \left(\frac{e^{-\frac{1}{\ep}}}{\ep^4}\right)^h \sup_{(n,y)\in Q}\BP_y\left(S_n=z \right),
\end{aligned}
$$
where $Q=[n_0-2R^2, n_0]\times B(x_{n}^\ep, 2R)$. Moreover, by Local Central Limit Theorem, see Theorem 2.1.1 of \cite{lawler2010random} for example, there is a finite constant $C<\infty$ independent to $n$ such that 
$$
\sup_{(n,y)\in Q}\BP_y\left(S_n=z \right)\le \frac{C}{\ep^4 n^2},
$$
which implies that 
$$
\left(\frac{e^{-\frac{1}{\ep}}}{\ep^4}\right)^h \sup_{(n,y)\in Q}\BP_y\left(S_n=z \right)\le C e^{-\frac{h}{\ep}} \ep^{-4(1+h)} n^{-2}
$$
and that 
\beq
\label{total variation 2}
\begin{aligned}
&\left|\BP\left(\hat S_{1,\lfloor \ep^4 n^2 \rfloor}=z\right)-\BP\left(\hat S_{2,\lfloor \ep^4 n^2 \rfloor}=z\right) \right|\\
&\le \sum_{x_{n}^\ep} \left|\BP\left(\hat S_{1,\lfloor \ep^4 n^2 \rfloor}=z\Big|\bar S_{\bar T_n^\ep}=x_{n}^\ep\right)-\BP\left(\hat S_{2,\lfloor \ep^4 n^2 \rfloor}=z\Big| \bar S_{\bar T_n^\ep}=x_{n}^\ep\right) \right|\BP\left(\bar S_{\bar T_n^\ep}=x_{n}^\ep\right)\\
&\le C e^{-\frac{h}{\ep}} \ep^{-4(1+h)} n^{-2} \sum_{x_{n}^\ep} \BP\left(\bar S_{\bar T_n^\ep}=x_{n}^\ep\right) \\
&\le C e^{-\frac{h}{\ep}} \ep^{-4(1+h)} n^{-2}.
\end{aligned}
\eeq
Thus,
\beq
\label{eeeeeee}
\begin{aligned}
& \sum_{z\in B(0,2n)} \left|\BP\left(\hat S_{1,\lfloor \ep^4 n^2 \rfloor}=z\right)-\BP\left(\hat S_{2,\lfloor \ep^4 n^2 \rfloor}=z\right) \right|\\
&\le \sum_{z\in B(0,2n)} C e^{-\frac{h}{\ep}} \ep^{-4(1+h)} n^{-2} \\
&\le C e^{-\frac{h}{\ep}} \ep^{-4(1+h)}.
\end{aligned}
\eeq
Combining \eqref{total variation 1} and \eqref{eeeeeee} we have 
\beq
\label{total variation 3}
\BP\left( \mathcal{A}_3^c \right)\le d_{TV}\left(\hat S_{1,\lfloor \ep^4 n^2 \rfloor}, \hat S_{2,\lfloor \ep^4 n^2 \rfloor}\right)\le \frac{1}{2} \big( 2c e^{-\beta/\ep^4}+C e^{-\frac{h}{\ep}} \ep^{-4(1+h)} \big) .
\eeq
And by strong Markov property, 
\beq
\label{total variation 4}
\BP\left(\{\bar S_{\bar T_n^\ep}\in U^\ep_n\}\cap \mathcal{A}_3^c \right)\le \frac{1}{8 \ep n}\left( 2c e^{-\beta/\ep^4}+C e^{-\frac{h}{\ep}} \ep^{-4(1+h)} \right)\ll \frac{\ep}{n}
\eeq
for all sufficiently large $n$ and sufficiently small $\ep$. Thus the proof of this lemma is complete. 
\end{proof}

With Lemma \ref{lemma coupling 1}, the proof of Proposition \ref{Proposition coupling} is complete. \qed

\subsection{Proof of Lemma \ref{brownian}}
Let $M , M_0 \in \mathbb{Z}_{+}$ such that $M > M_0 >1$. By Strong Markov Property,
\begin{equation}
\begin{aligned}
& P_{(0,Mn)} (\tau_{[-\delta n, \delta n] \times \{0 \} } = \tau_{D_n } ) \\
& = \sum_{y \in \partial^{out} B(0, M_0 n)} P_{(0,Mn)} (\tau_{\partial^{out} B(0, M_0 n) } =y) P_y (\tau_{[-\delta n, \delta n] \times \{0 \} } = \tau_{D_n } ).
\end{aligned}
\end{equation}
So by law of total probability,
\begin{equation}
\begin{aligned}
& \min_{y \in \partial^{out} B(0, M_0 n)} P_y (\tau_{[-\delta n, \delta n] \times \{0 \} } = \tau_{D_n} ) \\
& \leq P_{(0,Mn)} (\tau_{[-\delta n, \delta n] \times \{0 \} } = \tau_{D_n } ) \\
& \leq \max_{y \in \partial^{out} B(0, M_0 n)} P_y (\tau_{[-\delta n, \delta n] \times \{0 \} } = \tau_{D_n } ).
\end{aligned}
\end{equation}
Notice that if we fix $n$,
$$
\lim_{M \rightarrow \infty} P_{(0,Mn)} (\tau_{[-\delta n, \delta n] \times \{0 \} } = \tau_{D_n } ) = \harm_{D_n} ([-\delta n, \delta n] \times \{0 \}),
$$
and thus
\begin{equation}
\begin{aligned}
& \min_{y \in \partial^{out} B(0, M_0 n)} P_y (\tau_{[-\delta n, \delta n] \times \{0 \} } = \tau_{D_n} ) \\
& \leq \harm_{D_n} ([-\delta n, \delta n] \times \{0 \}) \\
& \leq \max_{y \in \partial^{out} B(0, M_0 n)} P_y (\tau_{[-\delta n, \delta n] \times \{0 \} } = \tau_{D_n } ).
\end{aligned}
\end{equation}
Let $\{y_n : y_n \in \partial^{out} B(0, M_0 n) \}$ be a sequence of points in $\mathbb{Z}^2$. Note that $||y_n||_2 \rightarrow \infty$ as $n \rightarrow \infty$. By invariance principle,
$$
\limsup_{n \rightarrow \infty} P_{y_n} (\tau_{[-\delta n, \delta n] \times \{0 \} } = \tau_{D_n } ) \leq \sup_{z \in \partial \bar{B} (0, M_0)} P^{BM}_z (\tau_{[-\delta, \delta] \times \{0\} } = \tau_{[-1, 1] \times \{0\} }),
$$
where $P^{BM}_z$ is the law of a Brownian motion starting at the point $z \in \mathbb{R}^2$. Since the choice of $\{y_n\}$ is arbitrary,  
$$
\limsup_{n \rightarrow \infty} \max_{y \in \partial^{out} B(0, M_0 n)} P_y (\tau_{[-\delta n, \delta n] \times \{0 \} } = \tau_{D_n } ) \leq \sup_{z \in \partial \bar{B} (0, M_0)} P^{BM}_z (\tau_{[-\delta, \delta] \times \{0\} } = \tau_{[-1, 1] \times \{0\} }).
$$
Similarly,
$$
\liminf_{n \rightarrow \infty} \min_{y \in \partial^{out} B(0, M_0 n)} P_y (\tau_{[-\delta n, \delta n] \times \{0 \} } = \tau_{D_n } ) \geq \inf_{z \in \partial \bar{B} (0, M_0)} P^{BM}_z (\tau_{[-\delta, \delta] \times \{0\} } = \tau_{[-1, 1] \times \{0\} }).
$$
Note that
\begin{equation}
\begin{aligned}
& \lim_{M_0 \rightarrow \infty} \sup_{z \in \partial \bar{B} (0, M_0)} P^{BM}_z (\tau_{[-\delta, \delta] \times \{0\} }  = \tau_{[-1, 1] \times \{0\} }) \\
& = \lim_{M_0 \rightarrow \infty} \inf_{z \in \partial \bar{B} (0, M_0)} P^{BM}_z (\tau_{[-\delta, \delta] \times \{0\} }= \tau_{[-1, 1] \times \{0\} }) \\
& = \bharm_{[-1, 1] \times \{0\}} ([-\delta, \delta] \times \{0\}).
\end{aligned}
\end{equation}
Therefore,
$$
\lim_{n \rightarrow \infty} \harm_{D_n} ([-\delta n, \delta n] \times \{0 \}) = \bharm_{[-1, 1] \times \{0\}} ([-\delta, \delta] \times \{0\}).
$$
\qed

With Lemma \ref{brownian}, the proof of Theorem \ref{theorem harmonic limit} is complete. \qed

\subsection{Proof of Proposition \ref{lemma away}}

In order to prove 
$$
\lim_{n\to\infty} n \cdot \max_{y\in l^c_n} \BP_y\left(S_{\bar \tau_{A_n}}=x\right)=0
$$
we first recall that 
$$
l_n=\left[- \lfloor n^{\alpha_2}\rfloor, \lfloor n^{\alpha_2}\rfloor\right]\times\left\{ \lfloor n^{\alpha_1}\rfloor\right\},
$$ 
$\alpha_1=(1+\alpha)/2$, $\alpha_2=(7+\alpha)/8$, and that $$l_n^c=\partial^{in}_lBox(n)\cup\partial^{in}_rBox(n)\cup\partial^{in}_uBox(n)\setminus l_n.
$$
Thus for any point $y\in l^c_n$, define 
$$
T_y=\{\lfloor y^{(1)} /2\rfloor\}\times [0,\infty)
$$
to be the vertical line located in the exact midway between $0$ and $y$. Noting that $\tau_{T_y}<\tau_{x}$, by strong Markov property we have 
\beq
\label{away 2}
\begin{aligned}
\BP_y\left(S_{\bar \tau_{A_n}}=x\right)&=\sum_{z\in T_y} \BP_y\left(\tau_{T_y}<\bar \tau_{A_n}, \ S_{\tau_{T_y}}=z\right)\BP_z \left(S_{\bar \tau_{A_n}}=x \right)\\
&=\hspace{-0.25 in}  \sum_{z\in T_y, \ z^{(2)}\ge n^4}\hspace{-0.25 in}  \BP_y\left(\tau_{T_y}<\bar \tau_{A_n}, \ S_{\tau_{T_y}}=z\right)\BP_z \left(S_{\bar \tau_{A_n}}=x \right)\\
&+\hspace{-0.225 in}  \sum_{z\in T_y, \ z^{(2)}< n^4}\hspace{-0.25 in}  \BP_y\left(\tau_{T_y}<\bar \tau_{A_n}, \ S_{\tau_{T_y}}=z\right)\BP_z \left(S_{\bar \tau_{A_n}}=x \right)\\
&\le \BP_y\left(\tau_{T_y}<\bar \tau_{A_n}, \ S^{(2)}_{T_y} \ge n^4 \right)\\
&+\hspace{-0.2 in}\max_{z\in T_y, \ z^{(2)}< n^4}\hspace{-0.2 in} \BP_z \left(S_{\bar \tau_{A_n}}=x \right) \BP_y\left(\tau_{T_y}<\bar \tau_{A_n}\right). 
\end{aligned}
\eeq
To control the right hand side of \eqref{away 2}, we first define 
$$
\bar D_n= \Big\{ T_y \cup [\lfloor y/2\rfloor,\infty)\times \{0\}  \Big\} \cap B(y,n^4)
$$
and then note that 
$$
\BP_y\left(\tau_{T_y}<\bar \tau_{A_n}, \ S^{(2)}_{T_y}\ge n^4 \right)\le \BP_y\left(\tau_{\partial^{out}B(y,n^4)}<\tau_{\bar D_n}\right). 
$$
Moreover, it is easy to see that 
$$
\text{rad}(\bar D_n)\ge n^4/2
$$
for $n$ sufficiently large, and that 
$$
d(\bar D_n, y) \leq \lfloor n^{\alpha_1}\rfloor.
$$
We apply Theorem $1$ in \cite{lawler2004beurling} with $\kappa =1$ and $A = \bar{D}_n$ on the discrete ball $B(y, n^4)$, then there exists a constant $C >0$ such that 
\beq
\label{away 3}
\BP_y\left(\tau_{\partial^{out}B(y,n^4)}<\tau_{\bar D_n}\right) \le \BP_y\left(\tau_{\partial^{out}B(y,n^4)}<\tau_{\bar D_{n_{[n^{\alpha_1}, n^4/2]}}}\right) \le C\sqrt{\frac{n^{\alpha_1}}{n^4}}=o\left(\frac{1}{n}\right). 
\eeq
Note that this is a Beurling estimate for random walk. And for the second term in the right hand side of \eqref{away 2}, note that for 
$$
\tilde D_n=L_0 \cap B\left(y,n^{\alpha_2}/2\right)
$$
we have 
\beq
\label{away 4}
\left\{\tau_{T_y}<\bar \tau_{A_n} \right\}\subset \left\{\tau_{\partial^{out} B\left(y,n^{\alpha_2}/2\right)}<\bar \tau_{\tilde D_n} \right\}
\eeq
Using again the Theorem $1$ of \cite{lawler2004beurling} to the right hand side of \eqref{away 4} we have 
\beq
\label{away 5}
\begin{aligned}
\BP_y\left(\tau_{T_y}<\bar \tau_{A_n}\right)&\le \BP_y \left(\tau_{\partial^{out} B\left(y,n^{\alpha_2}/2\right)}<\bar \tau_{\tilde D_n} \right)\le C n^{-(\alpha_2-\alpha_1)/2}. 
\end{aligned}
\eeq
At the same time, for any $z\in T_y$ such that $z^{(2)}< n^4$, again by the reversibility of simple random walk we have 
\beq
\label{away 6}
\begin{aligned}
\BP_z \left(S_{\bar \tau_{A_n}}=x \right)&=\sum_{n=1}^\infty \BP_z \left(S_1,S_2,\cdots, S_{n-1}\notin A_n, \ S_n=x \right)\\
&=\sum_{n=1}^\infty \BP_x \left(S_1,S_2,\cdots, S_{n-1}\notin A_n, \ S_n=z \right)\\
&=E_x\left[\# \text{ of visits to }z \text{ in } [0,\tau_{A_n})\right]\\
&=\BP_x\left(\tau_{z}<\tau_{A_n} \right)E_z\left[\# \text{ of visits to }z \text{ in } [0,\tau_{A_n})\right]\\
&=\frac{\BP_x\left(\tau_{z}<\tau_{A_n} \right)}{\BP_z\left(\tau_{A_n}<\tau_{z} \right)}. 
\end{aligned}
\eeq
To control the right hand side of \eqref{away 6}, we first refer to the well known result:
\begin{lemma}(Lemma 1 of  \cite{kesten1987hitting})
\label{Lemma 1 1987}
The series 
\beq
\label{discrete Green}
a(x)=\sum_{n=0}^\infty [P_0(S_n=0)-P_0(S_n=x)]
\eeq
converge for each $x\in \ZZ^2$, and the function $a(\cdot)$ has the following properties:

\beq
\label{property_a_1}
a(x)\ge 0, \ \forall x\in \ZZ^2, \ a(0)=0,
\eeq 

\beq
\label{property_a_2}
a\big((\pm 1,0)\big)=a\big((0,\pm 1)\big)=1
\eeq 

\beq
\label{property_a_3}
E_x[a(S_1)]-a(x)=\delta(x,0),
\eeq 
so $a(S_{n\wedge \tau_v}-v)$ is a nonnegative martingale, where $\tau_v=\tau_{\{v\}}$, for any $v\in \ZZ^2$. And there is some suitable $c_0$ such that
\beq
\label{property_a_4}
\left|a(x)-\frac{1}{2\pi} \log\|x\|-c_0 \right|=O(\|x\|^{-2}),
\eeq
as $\|x\|\to\infty$. 

\end{lemma}
Now we prove the following lower bound on the denominator:
\begin{lemma}
\label{lemma low denominator}
There is a finite constant $C<\infty$ such that for any nonzero $x\in \ZZ^2$, 
$$
\BP_0(\tau_x<\tau_0)\ge \frac{C}{(\log\|x\|)^2}. 
$$
\end{lemma}

\begin{proof}
First, it suffices to show this lemma for all $x$ sufficiently far away from $0$. We consider stopping time 
$$
\Gamma=\tau_{0}\wedge \tau_{\|x\|/2},
$$
By Lemma \ref{Lemma 1 1987}, we have 
$$
1=E_0\left[a(S_\Gamma)\big|  \tau_{\|x\|/2}<\tau_0\right]\BP_0\left(  \tau_{\|x\|/2}<\tau_0\right). 
$$
Thus by \eqref{property_a_4}, 
\beq
\label{denominator 1}
\BP_0\left(  \tau_{\|x\|/2}<\tau_0\right)=\frac{1}{E_0\left[a(S_\Gamma)\big|  \tau_{\|x\|/2}<\tau_0\right]}\ge \frac{\pi}{\log\|x\|}
\eeq
for all $x$ sufficiently far away from $0$. By strong Markov property, 
\beq
\label{denominator 1_5}
\begin{aligned}
\BP_0(\tau_x<\tau_0)&=\hspace{-0.2 in}\sum_{y\in \partial^{out}B(0,\|x\|/2)}\hspace{-0.2 in}\BP_0\left( \tau_{\|x\|/2}<\tau_0, \ S_{\tau_{\|x\|/2}}=y\right)\BP_y\left(\tau_x<\tau_0\right)\\
&\ge \frac{\pi}{\log\|x\|}  \min_{y\in \partial^{out}B(0,\|x\|/2)} \BP_y\left(\tau_x<\tau_0\right). 
\end{aligned}
\eeq
At the same time, for stopping time $\Gamma_1=\tau_{ \partial^{out} B(x,\|x\|/3)}$, and $\Gamma_2=\tau_{ \partial^{out} B(x,\|x\|/2)}$, we have 
\beq
\label{denominator 2}
\BP_y\left(\tau_x<\tau_0\right)\ge \hspace{-0.2 in}\sum_{z\in \partial^{out} B(x,\|x\|/3)} \hspace{-0.2 in}\BP_y\left(\Gamma_1<\tau_{\|x\|/3} , \ S_{\Gamma_1}=z\right)\BP_z\left(\tau_x<\Gamma_2 \right). 
\eeq
For the right hand side of \eqref{denominator 2}, we have by translation invariance of simple random walk,
$$
\BP_z\left(\tau_x<\Gamma_2 \right)=\BP_{z-x}\left(\tau_0<\tau_{\|x\|/2} \right).
$$
Moreover,
$$
\left[1-\BP_{z-x}\left(\tau_0<\tau_{\|x\|/2} \right)\right]E_{z-x}\left[a(S_{\Gamma})\big| \tau_{\|x\|/2}<\tau_0\right]=a(z-x),
$$
which implies that 
\beq
\label{denominator 3}
\BP_{z-x}\left(\tau_0<\tau_{\|x\|/2} \right)=\frac{E_{z-x}\left[a(S_{\Gamma})\big| \tau_{\|x\|/2}<\tau_0\right]-a(z-x)}{E_{z-x}\left[a(S_{\Gamma})\big| \tau_{\|x\|/2}<\tau_0\right]}. 
\eeq
Again, by Lemma  \ref{Lemma 1 1987}, we have that there are positive constants $c, C\in (0,\infty)$ such that uniformly for all $n$, $x$ and $z$ defined above,  
$$
E_{z-x}\left[a(S_{\Gamma})\big| \tau_{\|x\|/2}<\tau_0\right]-a(z-x)\ge c, 
$$
while 
$$
E_{z-x}\left[a(S_{\Gamma})\big| \tau_{\|x\|/2}<\tau_0\right]\le C \log\|x\|. 
$$
Thus we have
\beq
\label{denominator 4}
\BP_z\left(\tau_x<\Gamma_2 \right)=\BP_{z-x}\left(\tau_0<\tau_{\|x\|/2} \right)\ge \frac{c}{\log\|x\|}
\eeq
uniformly for all $n$, $x$ and $z$ defined above. 

On the other hand,  by invariance principle, there is a constant $c>0$ such that for any $y\in \partial^{out}B(0,\|x\|/2)$,
$$
\BP_y\left(\Gamma_1<\tau_{\|x\|/3} \right)\ge c. 
$$
Thus,
\beq
\label{denominator 5}
\BP_y\left(\tau_x<\tau_0\right)\ge \hspace{-0.2 in}\sum_{z\in \partial^{out} B(x,\|x\|/3)} \hspace{-0.2 in}\BP_y\left(\Gamma_1<\tau_{\|x\|/3} , \ X_{\Gamma_1}=z\right)\BP_z\left(\tau_x<\Gamma_2 \right)\ge \frac{c}{\log\|x\|}. 
\eeq
Now combining, \eqref{denominator 1}, \eqref{denominator 1_5}, and  \eqref{denominator 5}. The proof of this lemma is complete. \end{proof}

With Lemma \ref{lemma low denominator}, we look back at the right hand side of \eqref{away 6}. Noting that for any $z\in T_y$, $\tau_{T_y}\le\tau_z$ and that $\tau_{A_n}\le \tau_{D_n}$, we give the following upper bound estimate on its numerator: 
\begin{lemma}
\label{lemma up numerator}
Recall that $\alpha_2=(7+\alpha)/8$. Then for each $x\in A$, 
\beq
\BP_x\left(\tau_{T_y}<\tau_{D_n} \right)\le \frac{c}{n^{\alpha_2}}
\eeq
for all sufficiently large $n$ and all $y\in l_n^c$. 
\end{lemma}
\begin{proof}
For any given $x\in A$, define $x_0=(x^{(1)},0)$ be the projection of $x$ on $L_0$. Note that $x_0$ and $x$ are connected by a path independent to $n$, which implies that there is a constant $c>0$ also independent to $n$ such that 
$$
\BP_{x_0}\left(\tau_{T_y}<\tau_{D_n} \right)\ge c \BP_x\left(\tau_{T_y}<\tau_{D_n} \right).
$$
Thus to prove Lemma \ref{lemma up numerator} it suffices to replace $x$ by $x_0$. Moreover, recall that  $l_n^c=\partial^{in}_lBox(n)\cup\partial^{in}_rBox(n)\cup\partial^{in}_uBox(n)\setminus l_n$. For any $y\in l_n^c$, by the translation invariance of simple random walk, we have
$$
\BP_{x_0}\left(\tau_{T_y}<\tau_{D_n} \right)\le \BP_0\left(\tau_{I_{\lfloor n^{\alpha_2}/4\rfloor}}<\tau_{D_n} \right).
$$
Here recall the definition of $I_n$ in \eqref{set A_n}. Now by lemma \ref{lemma combinatorial},
$$
\BP_0\left(\tau_{I_{\lfloor n^{\alpha_2}/4\rfloor}}<\tau_{D_n} \right)\le \frac{C}{\lfloor n^{\alpha_2}/4\rfloor}
$$
and the proof of this lemma is complete. 
\end{proof}
Now apply \eqref{away 5}, \eqref{away 6}, Lemma \ref{lemma low denominator}, and Lemma \ref{lemma up numerator} together to the last term of \eqref{away 2}, we have 
$$
\begin{aligned}
\hspace{-0.2 in}\max_{z\in T_y, \ z^{(2)}< n^4}\hspace{-0.2 in} \BP_z \left(S_{\bar \tau_{A_n}}=x \right) \BP_y\left(\tau_{T_y}<\bar \tau_{A_n}\right)&\le C n^{-\alpha_2-(\alpha_2-\alpha_1)/2}(\log n)^2\\
&\le Cn^{-\frac{17}{16}+\frac{\alpha}{16}}(\log n)^2\ll n^{-1} 
\end{aligned}
$$
for all sufficiently large $n$. Thus, the proof of Proposition \ref{lemma away} is complete. \qed

\subsection{Proof of Proposition \ref{lemma stationary harmonic}}

To show 
$$
\lim_{n\to\infty} \sum_{y\in l_n}  \BP_y\left(S_{\bar \tau_{A_n}}=x\right)= \sharm_A(x), 
$$
we first prove that 
\begin{lemma}
\label{lemma away 2}
For any $x\in A$ and the truncations $A_n$ defined in \eqref{truncation}
\beq
\label{stationary harmonic}
\lim_{n\to\infty} \sum_{y\in l_n}  \BP_y\left(S_{\bar \tau_{A}}=x\right)= \sharm_A(x). 
\eeq
\end{lemma}
\begin{proof}
Recall that by definition that 
$$
 \sharm_A(x)=\lim_{k\to\infty} \sum_{z\in L_k} \BP_z\left(S_{\bar \tau_{A}}=x\right)
$$
and that 
$$
l_n=\left[- \lfloor n^{\alpha_2}\rfloor, \lfloor n^{\alpha_2}\rfloor\right]\times\left\{ \lfloor n^{\alpha_1}\rfloor\right\}.
$$
Thus 
$$
\lim_{n\to\infty} \sum_{z\in L_{\lfloor n^{\alpha_1}\rfloor}} \BP_z\left(S_{\bar \tau_{A}}=x\right)=  \sharm_A(x),
$$
while in order to prove Lemma \ref{lemma away 2}, it suffices to show that 
\beq
\label{away 2_1}
\lim_{n\to\infty} \sum_{z\in L_{\lfloor n^{\alpha_1}\rfloor}\setminus l_n} \BP_z\left(S_{\bar \tau_{A}}=x\right)= 0.
\eeq
Apply reversibility of simple random walk on each $z\in L_{\lfloor n^{\alpha_1}\rfloor}\setminus l_n$, we have 
\beq
\label{away 2_2}
\begin{aligned}
\sum_{z\in L_{\lfloor n^{\alpha_1}\rfloor}\setminus l_n} \BP_z\left(S_{\bar \tau_{A}}=x\right)&= E_x\left[\# \text{ of visits to }   L_{\lfloor n^{\alpha_1}\rfloor}\setminus l_n \text{ in } [0,\bar\tau_{A})\right]\\
&\le \frac{\BP_x\left(\tau_{  L_{\lfloor n^{\alpha_1}\rfloor}\setminus l_n }<\tau_{L_0} \right)}{\displaystyle \min_{z\in   L_{\lfloor n^{\alpha_1}\rfloor}\setminus l_n }\BP_z\left(\tau_{L_0}<\tau_{  L_{\lfloor n^{\alpha_1}\rfloor}\setminus l_n }\right)}. 
\end{aligned}
\eeq
First, for the denominator of \eqref{away 2_2}, note that 
$$
\tau_{  L_{\lfloor n^{\alpha_1}\rfloor} }\le \tau_{  L_{\lfloor n^{\alpha_1}\rfloor}\setminus l_n }
$$ 
We have for any $z\in   L_{\lfloor n^{\alpha_1}\rfloor}\setminus l_n$
\beq
\label{away 2_3}
\BP_z\left(\tau_{L_0}<\tau_{  L_{\lfloor n^{\alpha_1}\rfloor}\setminus l_n }\right)\ge \BP_z\left(\tau_{L_0}<\tau_{  l_{\lfloor n^{\alpha_1}\rfloor} }\right)\ge \frac{c}{\lfloor n^{\alpha_1}\rfloor}. 
\eeq
On the other hand, using exactly the same argument as in the proof of Lemma \ref{lemma up numerator}
\beq
\label{away 2_4}
\BP_x\left(\tau_{  L_{\lfloor n^{\alpha_1}\rfloor}\setminus l_n }<\tau_{L_0} \right)\le \frac{C}{\lfloor n^{\alpha_2}\rfloor}.
\eeq
Thus, combining \eqref{away 2_2}-\eqref{away 2_4}, the proof Lemma \ref{lemma away 2} is complete. 
\end{proof}

Now with Lemma \ref{lemma away 2}, it suffices to prove that 
\beq
\label{stationary harmonic 2}
\lim_{n\to\infty} \sum_{y\in l_n}  \left[\BP_y\left(S_{\bar \tau_{A_n}}=x\right)-\BP_y\left(S_{\bar \tau_{A}}=x\right) \right]=0. 
\eeq
Again by reversibility, 
$$
\BP_y\left(S_{\bar \tau_{A_n}}=x\right)=E_x\left[\# \text{ of visits to }y \text{ in }[0, \tau_{A_n}) \right]
$$
and 
$$
\BP_y\left(S_{\bar \tau_{A}}=x\right)=E_x\left[\# \text{ of visits to }y \text{ in }[0, \tau_{A}) \right],
$$
which implies that for each $y$
$$
\BP_y\left(S_{\bar \tau_{A_n}}=x\right)-\BP_y\left(S_{\bar \tau_{A}}=x\right)=E_x\left[\# \text{ of visits to }y \text{ in }[\tau_{A}, \tau_{A_n}) \right]
$$
and that 
\beq
\label{stationary harmonic 2_5}
\begin{aligned}
& \sum_{y\in l_n}  \left[\BP_y\left(S_{\bar \tau_{A_n}}=x\right)-\BP_y\left(S_{\bar \tau_{A}}=x\right) \right]\\
&=E_x\left[\# \text{ of visits to }l_n \text{ in }[\tau_{A}, \tau_{A_n}) \right].
\end{aligned}
\eeq
Here we use the natural convention that the number of visits equals to 0 over an empty interval. Moreover, define $\bar T_n=\{-n,n\}\times [0,\infty)$ and 
$$
\Gamma_4=\inf\{n> \tau_A, \ S_n \in \bar T_n \}. 
$$
Noting that 
$$
\{\tau_{A}<\Gamma_4< \tau_{A_n}\}\subset \{\tau_{A}< \tau_{A_n}\}\subset \{\tau_{\bar T_n}<\tau_{A_n} \},
$$
thus by strong Markov property, one can see that 
\beq
\label{stationary harmonic 3}
E_x\left[\# \text{ of visits to }l_n \text{ in }[\tau_{A}, \tau_{A_n}) \right]\le \frac{\BP_x\left(\tau_{\bar T_n}<\tau_{A_n}  \right)}{\displaystyle\min_{z\in l_n}\BP_z(\tau_{A_n}<\tau_{l_n})}. 
\eeq
First, for any $z=(z^{(1)},z^{(2)})\in l_n$, consider 
$$
(z^{(1)},0)+ \Big\{ [-\lfloor n^{\alpha_1}\rfloor, \lfloor n^{\alpha_1}\rfloor]\times [0,\lfloor n^{\alpha_1}\rfloor] \Big\} . 
$$
By Lemma \ref{lemma combinatorial} and translation/reflection invariance of simple random walk, 
\beq
\label{stationary harmonic 4}
\begin{aligned}
\BP_z(\tau_{A_n}<\tau_{l_n})&\ge \BP_0\left(\tau_{\partial^{in}_{u} I_{\lfloor n^{\alpha_1}\rfloor}}<\tau_{L_0}\right)\\
&\ge \BP_0\left(\tau_{\partial^{in}_{u} I_{\lfloor n^{\alpha_1}\rfloor}} =\tau_{\partial^{in} I_{\lfloor n^{\alpha_1}\rfloor}}\right)\\
&\ge \frac{1}{2} \BP_0\left(\tau_{\partial^{in} I_{\lfloor n^{\alpha_1}\rfloor}}<\tau_{L_0}\right)\\
&\ge \frac{1}{2} \BP_0\left(\tau_{L_{\lfloor n^{\alpha_1}\rfloor}}<\tau_{L_0}\right)= \frac{1}{8 \lfloor n^{\alpha_1}\rfloor}. 
\end{aligned}
\eeq
On the other hand, we have 
\beq
\label{stationary harmonic 5}
\begin{aligned}
\BP_x\left(\tau_{\bar T_n}<\tau_{A_n}  \right)&\le \BP_x\left(\tau_{\bar T_n}<\tau_{L_0}  \right)\\
&\le C\BP_0\left(\tau_{\partial^{in} I_{\lfloor n/2 \rfloor}}<\tau_{L_0}\right)\\
&\le 2C  \BP_0\left(\tau_{\partial^{in}_{u} I_{\lfloor n^{\alpha_1}\rfloor}}=\tau_{\partial^{in} I_{\lfloor n/2\rfloor}}\right)
&\le 2 C\BP_0\left(\tau_{L_{\lfloor n/2\rfloor}}<\tau_{L_0} \right)\le \frac{C}{n}. 
\end{aligned}
\eeq
Now combining \eqref{stationary harmonic 2_5}-\eqref{stationary harmonic 5}, we have shown \eqref{stationary harmonic 2} and the proof of Proposition \ref{lemma stationary harmonic} is complete. \qed

\subsection{Proof of Proposition \ref{lemma coupling two}}

At this point, in order to prove Theorem  \ref{theorem part 2}, we only need to show that for all sufficiently large $n$ and any $y\in l_n$, $2\harm_{Box(n)}(y)/\harm_{D_n}(0)$ can be arbitrarily close to one. First, for any $y\in l_n$, define 
$$
M(y,n)=n+\big|y^{(1)}\big|, \ \ m(y,n)=n-\big|y^{(1)}\big|.
$$
Recall that $Box(n)=[-n,n]\times \left[0,\lfloor n^{\alpha_1}\rfloor\right]$ and that $l_n=\left[- \lfloor n^{\alpha_2}\rfloor, \lfloor n^{\alpha_2}\rfloor\right]\times\left\{ \lfloor n^{\alpha_1}\rfloor\right\}$. We have
$$
n-\lfloor n^{\alpha_2}\rfloor\le m(y,n)\le n\le M(y,n)\le n+\lfloor n^{\alpha_2}\rfloor.
$$
Moreover, noting that 
$$
Box(n)\subset \left[y^{(1)}-M(y,n), y^{(1)}+M(y,n)\right]\times  \left[0,\lfloor n^{\alpha_1}\rfloor\right]
$$
and that 
$$
\left[y^{(1)}-m(y,n), y^{(1)}+m(y,n)\right]\times  \left[0,\lfloor n^{\alpha_1}\rfloor\right]\subset  Box(n),
$$
by definition we have 
$$
H_{\left[y^{(1)}-M(y,n), y^{(1)}+M(y,n)\right]\times  \left[0,\lfloor n^{\alpha_1}\rfloor\right]}(y)\le \harm_{Box(n)}(y)
$$
and 
$$
H_{\left[y^{(1)}-m(y,n), y^{(1)}+m(y,n)\right]\times  \left[0,\lfloor n^{\alpha_1}\rfloor\right]}(y)\ge \harm_{Box(n)}(y).
$$
Thus, combine translation invariance and Theorem \ref{theorem harmonic limit}, and note that for all $y\in l_n$, $M^{-1}(y,n)-n^{-1}=o(n^{-1})$, $m^{-1}(y,n)-n^{-1}=o(n^{-1})$. It is immediate to see that Proposition \ref{lemma coupling two} is equivalent to the following statement:
\begin{lemma}
\label{lemma difference}
For all integers $m,n>0$, define 
$$
\widehat{Box}(m,n)=[-n,n]\times [-m,0]. 
$$
For any $\ep>0$, we have 
\beq
\label{difference 1}
\harm_{D_n}(0)-2\harm_{\widehat{Box}(m,n)}(0) \in \left[0, \ \frac{\ep}{n}\right)
\eeq
for all sufficiently large $n$ and all $0<m\le 2 n^{\alpha_1}$. 
\end{lemma}
\begin{proof}
First, for the lower bound estimate, note that  
$$
D_n\subset \widehat{Box}(m,n)
$$
and that by the definition of harmonic measure, we have 
$$
\harm_{D_n}(0)=\lim_{k\to\infty} \BP_{(k,0)}\left(\tau_{D_n}=\tau_0 \right)
$$
and that 
$$
\harm_{\widehat{Box}(m,n)}(0)=\lim_{k\to\infty} \BP_{(k,0)}\left(\tau_{\widehat{Box}(m,n)}=\tau_0 \right).
$$
Moreover, by symmetry we have for all $k>n$, 
$$
\BP_{(k,0)}\left(\tau_{D_n}=\tau_0 \right)=2\BP_{(k,0)}\left(\tau_{D_n}=\tau_0, \ S_{\tau_0-1} =(0,1)\right).
$$
At the same time on can see that in the event $\left\{\tau_{\widehat{Box}(m,n)}=\tau_0 \right\}$, the random walk has to visit $0$ through $(0,1)$, which implies that 
$$
\BP_{(k,0)}\left(\tau_{D_n}=\tau_0, \ S_{\tau_0-1}=(0,1)\right)\ge  \BP_{(k,0)}\left(\tau_{\widehat{Box}(m,n)}=\tau_0 \right).
$$ 
Taking limit as $k\to\infty$, we have shown the lower bound estimate. For the upper bound estimate, again we note that for each sufficiently large $k$ and a random walk starting from $(k,0)$
\beq
\label{difference 2}
\begin{aligned}
&\left\{\tau_{D_n}=\tau_0, \ S_{\tau_0-1}=(0,1)\right\}\setminus \left\{\tau_{\widehat{Box}(m,n)}=\tau_0 \right\}\\
&=\left\{\tau_{D_n}=\tau_0, \ S_{\tau_0-1}=(0,1)\right\}\cap \left\{\tau_{\widehat{Box}(m,n)\setminus D_n}< \tau_{D_n}\right\},
\end{aligned}
\eeq
which, by strong Markov property implies that 
\beq
\label{difference 3}
\begin{aligned}
&\BP_{(k,0)}\left(\tau_{D_n}=\tau_0, \ S_{\tau_0-1}=(0,1)\right)-  \BP_{(k,0)}\left(\tau_{\widehat{Box}(m,n)}=\tau_0 \right)\\
\le& \max_{y\in \widehat{Box}(m,n)\setminus D_n} \BP_y\left(\tau_{(0,1)}<\tau_{D_n} \right). 
\end{aligned}
\eeq
Now in order to find the upper bound of the right hand side of \eqref{difference 3}, we consider the following two cases based on the location of point $y=(y^{(1)}, y^{(2)})\in \widehat{Box}(m,n)\setminus D_n$: 

\

\noindent \large {\bf Case 1:} \normalsize  

\begin{figure}[h]
\centering 
\begin{tikzpicture}[scale=0.5]
\tikzstyle{redcirc}=[circle,
draw=black,fill=red,thin,inner sep=0pt,minimum size=2mm]
\tikzstyle{bluecirc}=[circle,
draw=black,fill=blue,thin,inner sep=0pt,minimum size=2mm]

\draw [black,thick] (10,0) to (0,0);

\draw[black,thick] (10,0) to (20,0);

\draw [red,dashed] (0,0) to (0,-2);

\draw [red,dashed] (20,0) to (20,-2);

\draw[red,dashed] (0,-2) to (20,-2);

\node (v1) at (17,-1) {$Box(m,n)$};

\node (v1) at (11.2,0.5) {$(0,1)$};

\draw[blue,dashed] (15,0) arc (0:360:5) ;

\draw [red,thick] plot [smooth, tension=2.5] coordinates {(9,-2) (8, -1) (7,-3) (6.5,-2) (6,-3)};

\node (v5) at (6,-3) [bluecirc] {};

\draw[blue,dashed] (13, 0) to (13, 3); 

\draw[blue,dashed] (7, 0) to (7, 3); 

\draw[blue,dashed] (7, 3) to (13, 3);

\draw [black,dashed] plot [smooth, tension=2.5] coordinates {(6,-3) (-6,0)   (7,4)};

\draw [green,thick] plot [smooth, tension=2.5] coordinates {(10,0.5) (10.27, 2) (10.54, 1.5) (10.8,3) (11, 2.5) (9, 3.5) (8,3) (7.5,2) (7,4)};

\node (v5) at (7,4) [bluecirc] {};

\node (v0) at (10,0.5) [bluecirc] {};

\node (v0) at (10.8,3) [bluecirc] {};

\node (v5) at (9,-2) [bluecirc] {};

\node (v6) at (9.5,-2.5) {$y$};

\node (v6) at (16.9,5.2) {$\BP$};

\draw [green,thick] (17,4.9) to (17.5,4.9);

\node (v6) at (20.36,5.2) {$=O(n^{-1}\log n)$};

\node (v6) at (16.9,4.2) {$\BP$};

\draw [red,thick] (17,3.9) to (17.5,3.9);

\node (v6) at (20.4,4.2) {$=O(n^{-(1-\alpha)/2})$};

\end{tikzpicture}
\caption{Illustration of proof for Case 1}
\end{figure}
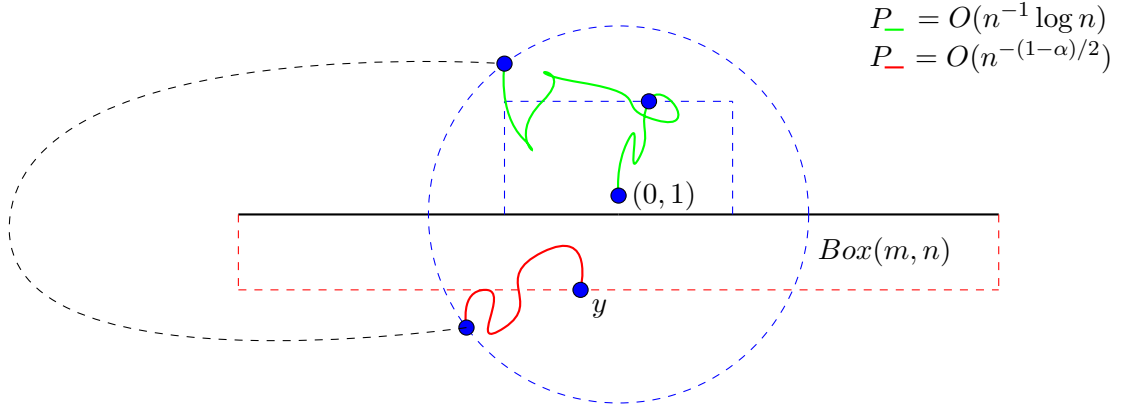

\mn If $\big| y^{(1)}\big|\le n/3$, for all nearest neighbor paths starting at $y$ which hit $(0,1)$ before $D_n$, they first have to hit $\partial^{out}B(0,n/2)$. Thus we have 
\beq
\label{difference 4}
\begin{aligned}
\BP_y\left(\tau_{(0,1)}<\tau_{D_n} \right)&=\hspace{-0.25 in}\sum_{z\in \partial^{out}B(0,n/2)} \hspace{-0.25 in}\BP_y\left(\tau_{n/2}<\tau_{D_n}, \ S_{\tau_{n/2}}=z \right)\BP_z\left(\tau_{(0,1)}<\tau_{D_n} \right)\\
&\le  \BP_y\left(\tau_{n/2}<\tau_{D_n}\right) \max_{z\in \partial^{out}B(0,n/2)}\BP_z\left(\tau_{(0,1)}<\bar \tau_{D_n} \right). 
\end{aligned}
\eeq
For the first term of the right hand side of \eqref{difference 4}, recalling that $d(y,D_n)=\big|y^{(2)}\big|=m\le 2n^{\alpha_1}$ and that $\big| y^{(1)}\big|<n/3$, we have by the same Beurling estimate, there exists a constant $C<\infty$ independent to the choice of $n,m$ and $y$ satisfying Case 1, such that 
\beq
\label{difference 4_5}
\BP_y\left(\tau_{n/2}<\tau_{D_n}\right)\le Cn^{-(1-\alpha_1)/2}. 
\eeq
At the same time, for any $z\in \partial^{out}B(0,n/2)$, to control the upper bound on $\BP_z\left(\tau_{(0,1)}<\bar \tau_{D_n} \right)$, one can concentrate on the upper half plane, since each path from $y$ to $(0,1)$ must pass through some point $z \in \partial^{out}B(0,n/2) \cap \{x \in \upspace : x^{(2)} >0 \}$. Now for any such $z$, by reversibility, we have 
\beq
\label{difference 5}
\begin{aligned}
\BP_z\left(\tau_{(0,1)}<\bar \tau_{D_n} \right)&=E_{(0,1)} \left[\# \text{ of visits to }z \text{ in }[0,\tau_{D_n\cup \{(0,1)\}})  \right]\le\frac{\BP_{(0,1)}\left(\tau_{z}<\bar \tau_{D_n} \right)}{\BP_z\left(\bar \tau_{D_n}<\tau_{z} \right)}.
\end{aligned}
\eeq
For the numerator, note that for all sufficiently large $n$, $[-\lfloor n/3 \rfloor,\lfloor n/3 \rfloor]\times [0,\lfloor n/3 \rfloor]\subset B(0,n/2)$. Applying the same argument as we repeatedly used in this paper, we have 
$$
\BP_{(0,1)}\left(\tau_{z}<\bar \tau_{D_n} \right)\le \frac{C}{n}. 
$$
At the same time, 
$$
\begin{aligned}
\BP_z&\left(\bar \tau_{D_n}<\tau_{z} \right)\\
&\ge \hspace{-0.26 in} \sum_{w\in \partial^{out} B(z, \frac{z^{(2)}}{2})} \hspace{-0.26 in} \BP_z\left(\bar \tau_{\partial^{out} B(z, \frac{z^{(2)}}{2})}<\tau_{z}, \bar \tau_{\partial^{out} B(z, \frac{z^{(2)}}{2})}=\tau_{w} \right) \BP_w\left(\bar \tau_{D_n}<\tau_{\partial^{out} B(z, \frac{z^{(2)}}{3})} \right). 
\end{aligned}
$$
And by invariance principle and the fact that $z^{(2)}\in (0,n]$, we have there is a constant $c>0$ independent to the choices of $n, z$ and $w$, such that 
$$
\BP_w\left(\bar \tau_{D_n}<\tau_{\partial^{out} B(z, \frac{z^{(2)}}{3})} \right)\ge c. 
$$
Thus by Lemma \ref{lemma low denominator},
$$
\BP_z\left(\bar \tau_{D_n}<\tau_{z} \right)\ge c \BP_z\left(\bar \tau_{\partial^{out} B(z, \frac{z^{(2)}}{2})}<\tau_{z}\right)\ge \frac{c}{ \big( \log \frac{z^{(2)}}{2} \big)^2 } \ge \frac{c}{(\log n)^2 },
$$
which by \eqref{difference 5} implies that 
\beq
\label{difference 6}
\begin{aligned}
\BP_z\left(\tau_{(0,1)}<\bar \tau_{D_n} \right)\le \frac{C (\log n)^2 }{n}.
\end{aligned}
\eeq
Now combining \eqref{difference 3}, \eqref{difference 4}, \eqref{difference 4_5}, and \eqref{difference 6}, 
\beq
\label{difference 7}
\begin{aligned}
\BP_y\left(\tau_{(0,1)}<\tau_{D_n} \right)\le C n^{-(3-\alpha_1)/2} (\log n)^2\ll n^{-1}
\end{aligned}
\eeq
and thus our lemma hold when $y$ in Case 1.

\noindent \large {\bf Case 2:} \normalsize  
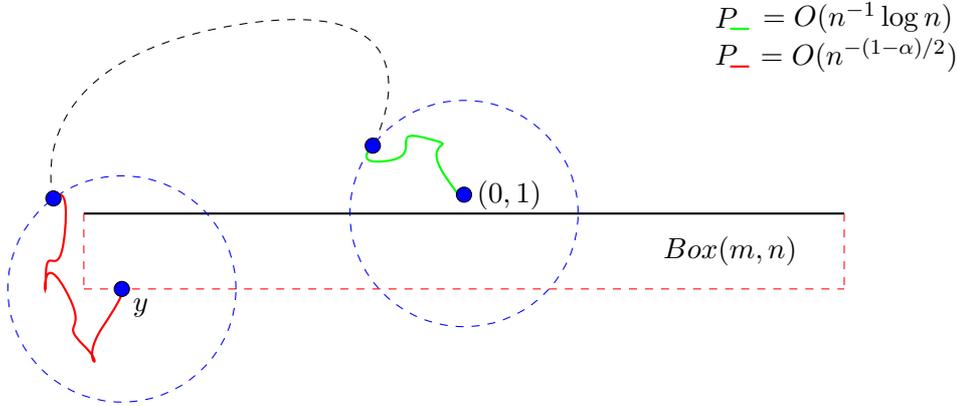
\begin{figure}[h]
\centering 
\begin{tikzpicture}[scale=0.5]
\tikzstyle{redcirc}=[circle,
draw=black,fill=red,thin,inner sep=0pt,minimum size=2mm]
\tikzstyle{bluecirc}=[circle,
draw=black,fill=blue,thin,inner sep=0pt,minimum size=2mm]

\draw [black,thick] (10,0) to (0,0);

\draw[black,thick] (10,0) to (20,0);

\draw [red,dashed] (0,0) to (0,-2);

\draw [red,dashed] (20,0) to (20,-2);

\draw[red,dashed] (0,-2) to (20,-2);

\draw[blue,dashed] (13,0) arc (0:360:3) ;

\draw[blue,dashed] (4,-2) arc (0:360:3) ;

\draw [black, dashed] plot [smooth, tension=2.5] coordinates {(1-1.8,-2+2.4) (4,5) (10-2.4,1.8)};

\draw [red, thick] plot [smooth, tension=2.5] coordinates {(1,-2) (0.5,-3) (0,-3.5) (-0.5, -2) (-1, -1.7) (-0.5,-0.5) (1-1.8,-2+2.4)};

\draw  [green, thick] plot [smooth, tension=2.5] coordinates {(10-2.4,1.8) (8,1.4) (9,2) (9.5, 1) (10,0.5)};

\node (v5) at (1,-2) [bluecirc] {};

\node (v1) at (17,-1) {$Box(m,n)$};

\node (v1) at (11.2,0.5) {$(0,1)$};

\node (v0) at (10,0.5) [bluecirc] {};

\node (v6) at (1.5,-2.5) {$y$};

\node (v0) at (1-1.8,-2+2.4) [bluecirc] {};

\node (v0) at (10-2.4,1.8) [bluecirc] {};

\node (v6) at (16.9,5.2) {$\BP$};

\draw [green,thick] (17,4.9) to (17.5,4.9);

\node (v6) at (20.36,5.2) {$=O(n^{-1}\log n)$};

\node (v6) at (16.9,4.2) {$\BP$};

\draw [red,thick] (17,3.9) to (17.5,3.9);

\node (v6) at (20.4,4.2) {$=O(n^{-(1-\alpha)/2})$};

\end{tikzpicture}
\caption{Illustration of proof for Case 2}
\end{figure}

\mn Otherwise, if $\big|y^{(1)}\big|>n/3$, our proof follows the same techniques on slightly different stopping times. Consider two neighborhoods: $B(0,\frac{n}{7})$ and $B(y,\frac{n}{7})$. It is easy to see that 
$$
\partial^{out}B\left(0,\frac{n}{7}\right) \cap \partial^{out}B\left(y,\frac{n}{7}\right) =\emptyset.
$$
Using the same argument as in Case 1, 
$$
\BP_y\left(\tau_{(0,1)}<\tau_{D_n} \right)=\hspace{-0.25 in}\sum_{w\in \partial^{out}B(y,\frac{n}{7})} \hspace{-0.25 in}\BP_y\left(\tau_{ \partial^{out}B(y,\frac{n}{7})}<\tau_{D_n}, \ S_{\tau_{ \partial^{out}B(y,\frac{n}{7})}}=w \right)\BP_w\left(\tau_{(0,1)}<\tau_{D_n} \right).
$$
Moreover for any $w\in \partial^{out}B(y,\frac{n}{7})$ the random walk starting at $w$ has to first visit $\partial^{out}B(0,\frac{n}{7})$ before ever reaches $(0,1)$. This implies that 
$$
\begin{aligned}
\BP_w\left(\tau_{(0,1)}<\tau_{D_n} \right)&=\hspace{-0.25 in}\sum_{z\in \partial^{out}B(0,\frac{n}{7})} \hspace{-0.25 in}\BP_w\left(\tau_{n/7}<\tau_{D_n}, \ S_{\tau_{n/7}}=z \right)\BP_z\left(\tau_{(0,1)}<\tau_{D_n} \right)\\
&\le \max_{z\in \partial^{out}B(0,\frac{n}{7})}\BP_z\left(\tau_{(0,1)}<\tau_{D_n} \right). 
\end{aligned}
$$
Thus we have 
\beq
\label{difference 8}
\BP_y\left(\tau_{(0,1)}<\tau_{D_n} \right)\le \BP_y\left(\tau_{ \partial^{out}B(y,\frac{n}{7})}<\tau_{D_n}\right) \max_{z\in \partial^{out}B(0,\frac{n}{7})}\BP_z\left(\tau_{(0,1)}<\tau_{D_n} \right). 
\eeq
Now since $y^{(2)}=-m\ge -2n^{\alpha_1}$, it is easy to see that 
$$
\text{rad}\left(B\left(y,\frac{n}{7}\right)\cap D_n\right)\ge \frac{n}{4}
$$
for all sufficiently large $n$. Thus by \eqref{difference 4_5} and  \eqref{difference 6}, there exists a constant $C<\infty$ independent to the choice of $n,m$ and $y$ satisfying Case 2, such that 
\beq
\label{difference 9}
 \BP_y\left(\tau_{ \partial^{out}B(y,\frac{n}{7})}<\tau_{D_n}\right)\le Cn^{-(1-\alpha_1)/2}. 
\eeq
and that 
\beq
\label{difference 10}
 \max_{z\in \partial^{out}B(0,\frac{n}{7})}\BP_z\left(\tau_{(0,1)}<\tau_{D_n} \right)\le \frac{C (\log n)^2 }{n}.
\eeq
Thus we also have 
\beq
\label{difference 11}
\begin{aligned}
\BP_y\left(\tau_{(0,1)}<\tau_{D_n} \right)\le C n^{-(3-\alpha_1)/2} (\log n)^2 \ll n^{-1}
\end{aligned}
\eeq
and thus our lemma hold when $y$ in Case 2 and the proof of Lemma \ref{lemma difference} is complete. 
\end{proof}
With Lemma \ref{lemma difference}, we have concluded the proof of Proposition  \ref{lemma coupling two}.  \qed

\bibliographystyle{plain}
\bibliography{eigen}

\end{document}